\documentclass[10pt]{article}

\usepackage[utf8]{inputenc}
\usepackage[T1]{fontenc}
\usepackage{amsmath,amssymb,amsthm}
\usepackage{mathtools}
\usepackage{geometry}
\usepackage{graphicx}
\usepackage{hyperref}
\usepackage{enumitem}
\usepackage{mathrsfs}
\usepackage{mwe}
\usepackage{booktabs}
\usepackage{arydshln}
\usepackage{color}
\usepackage{listings}
\usepackage{orcidlink}
\usepackage[vlined]{algorithm2e}

\theoremstyle{plain}
\newtheorem{theorem}{Theorem}[section]
\newtheorem{proposition}[theorem]{Proposition}
\newtheorem{lemma}[theorem]{Lemma}

\theoremstyle{definition}
\newtheorem{definition}[theorem]{Definition}
\theoremstyle{remark}
\newtheorem{remark}[theorem]{Remark}
\theoremstyle{definition}
\newtheorem{assumption}[theorem]{Assumption}

\newtheorem*{proof*}{Proof}

\title{
	A Numerical Investigation of the Rayleigh--Faber--Krahn Inequality for Polyhedral Domains
}

\author{
	Josu{\'e} D. D{\'i}az-Avalos%
	\thanks{
	\texttt{josue.diazavalos@uni-due.de},
	\orcidlink{0000-0002-2585-4315} (Corresponding author)
	}
	\and
	Antoine Laurain%
	\thanks{
		\texttt{antoine.laurain@uni-due.de},
		\orcidlink{0000-0002-8733-5190}
	}
}

\date{Fakult\"at f\"ur Mathematik, Universit\"at Duisburg-Essen, Germany}

\begin{document}

\maketitle

\begin{abstract}
	We present numerical results for the problem of minimizing
	the first Dirichlet eigenvalue of the Laplacian among
	three-dimensional convex polyhedra with a prescribed number of facets.
	This problem can be viewed as a polyhedral version of the classical Rayleigh--Faber--Krahn inequality.
	Using a parameterization of polytopes by supporting hyperplanes, the problem is first reformulated as an equivalent finite-dimensional constrained minimization
	problem.
	A Lagrangian framework is employed for the numerical solution.
	Under the assumption that each vertex is incident to exactly $d$ facets, the sensitivity analysis of the Lagrangian combines finite-dimensional perturbations of the facets of  $d$-dimensional polyhedra with infinite-dimensional shape derivatives of both the eigenvalue and the volume.
	Depending on the number of facets, the numerical results yield well-known polyhedra, such as
	Platonic solids, prisms, and the tetrakaidecahedron,
	or reveal nonregular optimal shapes exhibiting several symmetries.
\end{abstract}

\textbf{Keywords:}
shape derivative,
Rayleigh--Faber--Krahn Inequality,
Dirichlet eigenvalue,
convex polyhedra.

\section{Introduction}

Eigenvalue problems play a fundamental role in the analysis
of partial differential equations.
The spectrum of elliptic operators determines stability properties,
long-time behavior of solutions to evolution equations,
and natural vibration frequencies of physical systems.
Understanding the dependence of eigenvalues on the geometry
of the domain is therefore a classical problem in spectral theory.

In 1877, Lord Rayleigh conjectured that, among planar domains of fixed area,
the disk minimizes the first Dirichlet eigenvalue of the Laplacian.
This conjecture was proved independently by Georg Faber (1923) and Edgar Krahn (1924)~\cite{faber,krahn}.
P{\'o}lya and Szeg{\"o} considered the analogous problem in the class of planar polygons with
a fixed number of sides and fixed area, and conjectured that the regular polygon is the optimal shape.
The conjecture has been proved for polygons with $N=3,4$ sides using Steiner Symmetrization, see~\cite[Thm.~3.3.3]{MR2251558}.
For the cases $N\geq5$, significant progress has been achieved recently by Bogosel and Bucur who showed that proving 
the conjecture can be reduced to a finite number of certified numerical
computations~\cite{MR4683390}.
Building on this result, they proved in \cite{bogosel2024}, using certified finite element approximations,  that regular polygons are
local minimizers for $N=5,6$.

Other recent work related to spectral optimization in polygons include \cite{MR4742089},
where  an asymptotic expansion with respect to $N$ is derived for the first Dirichlet eigenvalue of a regular $N$-gon of area $\pi$.
In \cite{MR2264470}, they numerically  investigate new types of bounds for the first Dirichlet eigenvalue of polygons.
In particular, they conjecture that the first Dirichlet eigenvalues are monotonically decreasing in $N$ and confirm the monotonicity by numerical experiments.
The author in \cite{MR4729687} explicitly constructs a set of $N$-gons such that the first eigenvalue is minimized by the
regular convex $N$-gon,  in the collection of $N$-gons with the same area, for all sufficiently large $N$.
Polygons are also solutions of maximization of the Dirichlet eigenvalue with a convexity constraint, as shown in \cite{MR4726519}.
We also refer to the survey \cite{Henrot2017} for various results on spectral problems in polygonal domains.

Motivated by these recent advances, and by the fact that the Rayleigh--Faber--Krahn inequality holds in any dimension,
it is natural to consider the analogous problem for polyhedral domains in three and higher dimensions.
While a few references on spectral optimization provide numerical results in three dimensions, they usually focus on smooth sets. 
We also mention \cite{Greif}, where the eigenvalues and eigenfunctions of the Laplacian on the surfaces of regular polyhedra are investigated.
In this work, we present numerical results for the Rayleigh--Faber--Krahn problem restricted to polyhedral domains in three dimensions. 
This study also continues the line of research in \cite{MR5011129}, where numerical results were obtained for the polyhedral version of the Saint--Venant inequality.
Our aim is to stimulate interest in this question, provide evidence for possible optimal shapes, and offer a basis for new conjectures and further theoretical investigation.

Let $\mathcal{O}$ be the set of open bounded subsets of $\mathbb{R}^d$.
Let $\lambda\colon\mathcal{O}\rightarrow\mathbb{R}$ be the shape functional
\[
	\lambda(\Omega)\coloneqq
	\min_{u\in H_0^1(\Omega)\setminus\{0\}}
	\frac{\displaystyle\int_\Omega|\nabla u|^2}{\displaystyle\int_\Omega u^2}.
\]
It is well-known that $\lambda(\Omega)$ is the first Dirichlet eigenvalue
of the Laplacian corresponding to $\Omega$. 
If $\Omega$ is connected, then $\lambda(\Omega)$ is a simple eigenvalue
and a positive number. 
Let $B\subset\mathbb{R}^d$ be the Euclidean ball.
The classical Rayleigh--Faber--Krahn inequality states that
\begin{equation}\label{eq:FK}
	|B|^{2/d}\lambda(B)\leq|\Omega|^{2/d}\lambda(\Omega)
	\quad\forall\,\Omega\in\mathcal{O}.
\end{equation}
Let $\mathcal{K}$ be a finite set of indices with cardinality $|\mathcal{K}|>d$.
Let $\mathcal{P}\subset\mathcal{O}$ be the class of (convex) polytopes in
$\mathbb{R}^d$, with $|\mathcal{K}|$ facets and dimension $d$.
In this work, we provide numerical results for the following problem:
\begin{equation}\label{eq:polyFK}
	P^*\in\mathcal{P}:\quad
	|P^*|^{2/d}\lambda(P^*)\leq|P|^{2/d}\lambda(P)
	\quad \forall\,P\in\mathcal{P},
	\qquad
	\text{for }d=3.
\end{equation}
To the best of our knowledge, there are no previous works
showing numerical results for problem \eqref{eq:polyFK}
in dimensions $d\geq3$.

In Section~\ref{sec:sp} we recall the shape derivatives
of the first Dirichlet eigenvalue of the Laplacian and the volume,
with emphasis on convex domains.
In Section~\ref{sec:main}, polytopes and their perturbation are treated. 
The optimization problem to be solved is stated in Section~\ref{sec:opt_prb}.
In Section~\ref{sec:num_mtd}, we describe the numerical method and
apply the results of the previous sections.
Finally, the numerical results are provided in Section~\ref{sec:num}.
We provide an appendix with \texttt{GeoGebra} commands
for reproducing the optimal shapes.


\section{Shape derivative}\label{sec:sp}

In this section, we recall the expressions of the shape derivatives of
the volume and of the first Dirichlet eigenvalue of the Laplacian.
We follow the shape calculus framework of Henrot and Pierre \cite{MR3791463},
which is based on domain perturbations induced by
a smooth family of diffeomorphisms.

\begin{theorem}\label{thm:sp}
	Let $\Omega\in\mathcal{O}$ be a convex domain.
	Let $\mathrm{Id}$ denote the identity mapping.
	Let
	\[
		\{\Phi_t\}_{t\in[0,\tau)}, \quad
		\Phi_t\colon\mathbb{R}^{d}\to\mathbb{R}^{d}, \quad
		\text{with}\;\tau>0,
	\]
	be a family of maps such that
	$\Phi_{t}-\mathrm{Id}\in W^{1,\infty}(\mathbb{R}^d;\mathbb{R}^d)$,
	$t\mapsto\Phi_t$ is differentiable at $t=0$, and $\Phi_0=\mathrm{Id}$.
	Set
	\[
		\Omega_t\coloneqq\Phi_t(\Omega)
		\quad\text{and}\quad
		\theta\coloneqq\frac{d}{dt}\Phi_t\Big|_{t=0}.
	\]
	The functions
	$t\mapsto|\Omega_t|$ and
	$t\mapsto\lambda(\Omega_t)$
	are differentiable at $t=0$, with derivatives
	\begin{equation}\label{eq:sp_vol}
		\frac{d}{dt}|\Omega_t|\Big|_{t=0} =
		\int_{\partial\Omega} \theta\cdot\nu
	\end{equation}
	and
	\begin{equation}\label{eq:sp_eig}
		\frac{d}{dt}\lambda(\Omega_t)\Big|_{t=0}=
		-\int_{\partial\Omega} (\partial_\nu u)^2 \theta\cdot\nu,
	\end{equation}
	where $\nu$ denotes the outward unit normal vector to $\partial\Omega$ and
	$u\in H_0^1(\Omega)$ is the $L^2$-normalized eigenfunction associated
	to $\lambda(\Omega)$, that is,
	\begin{equation}\label{eq:eig_pb}
		\int_\Omega u^2=1
		\quad\text{and}\quad
		\int_\Omega\nabla u\cdot\nabla v =
		\lambda(\Omega)\int_\Omega uv \quad\forall v\in H_0^1(\Omega).
	\end{equation}
	The derivatives \eqref{eq:sp_vol} and \eqref{eq:sp_eig}  are called shape derivatives.
\end{theorem}

This result can be found in \cite[Thm. 2.5.1 and 2.5.3]{MR2251558}.
The assumption that $\Omega$ is convex implies that $\Omega_t$ is connected,
and hence $\lambda(\Omega_t)$ remains simple, for small values of $t$.
Moreover, since $\Omega$ is convex, $u\in H^2(\Omega)$;
see for instance \cite[Thm.~3.2.1.2]{MR3396210}.

Alternatively, one can also compute the shape derivative of the first Dirichlet eigenvalue in {\it distributed} or {\it weak form}, as
\begin{equation}\label{eq:speigdistrib}
	\frac{d}{dt}\lambda(\Omega_t)\Big|_{t=0}=
	\int_{\Omega} S_1\colon\! D\theta,
\end{equation}
where
$S_1 \coloneqq (|\nabla u|^2 - \lambda(\Omega) u^2) I -2\nabla u\otimes \nabla u $
is a matrix-valued function and
$I$ denotes the identity matrix; see the proof of \cite[Thm.~4]{MR4572197} or \cite[Thm.~2.1]{MR4683390}, in which the second-order shape derivative in distributed form is also computed.
Note that, under appropriate regularity assumptions on the set $\Omega$, the boundary expression~\eqref{eq:sp_eig} and the distributed expression \eqref{eq:speigdistrib}
are equal and related through the divergence theorem, see \cite[Prop.~1 and 2]{MR4053037}. 
In this work, we only use the boundary expressions \eqref{eq:sp_vol} and \eqref{eq:sp_eig},
but the distributed shape derivative can also be employed for numerical purposes in combination with the expression of $D\theta$ given in \eqref{eq:tht_Dtht}.

\section{Perturbation of polyhedral domains}\label{sec:main}

We begin with the standard definitions for polytopes;
see \cite[Ch.~2]{MR0683612} for more details.

\begin{definition}\label{def:ph_set}
	A subset $P\subset\mathbb{R}^{d}$ is called \emph{polyhedral set} if it is
	the intersection of finitely many closed half-spaces.
	A polyhedral set $P$ is \emph{irreducible} if the number of half-spaces
	that determine $P$ cannot be reduced.
	Given a polyhedral set $P$, a subset $F$ of $P$ is called a \emph{face}
	of $P$ if it is the intersection of $P$ with a supporting
	hyperplane of $P$.
	The faces of dimension\footnote{
		The dimension of a set is understood as the dimension of
		the smallest affine space containing it.
	}
	$0$ and $\dim P-1$ are called \emph{vertices} and \emph{facets},
	respectively.
	If $P$ is an irreducible polyhedral set of dimension $d$, a subset $F$ of $P$
	is a facet if and only if it is the intersection of $P$ with one of its
	half-spaces.
\end{definition}

\begin{definition}\label{def:polytope}
	A \emph{polytope} is a subset of $\mathbb{R}^{d}$ which is the convex hull
	of a non-empty finite set of points.
	A non-empty subset $P\subset\mathbb{R}^{d}$ is a polytope if and only if it
	is a bounded polyhedral set.
	A polytope $P$ of dimension $d$ is said to be \emph{simple} if each vertex
	of $P$ is contained in exclusively $d$ facets of $P$.
\end{definition}

Note that polytopes are thus closed sets.
For our purposes, it is more convenient to work with open sets, hence the relative interior of a polytope $P$, as defined in Definition~\ref{def:polytope}, will also be called polytope.
As defined in the introduction, $\mathcal{P}\subset\mathcal{O}$
denotes the set of (open) polytopes of dimension $d$ in $\mathbb{R}^d$
(recall that $\mathcal{O}$ is the set of open bounded subsets of $\mathbb{R}^d$).
Without loss of generality, we assume that every polytope in $\mathcal{P}$
contains the origin and that any representation as a polyhedral set is irreducible.

\begin{definition}\label{def3.3}
	Let $P\in\mathcal{P}$.
	A family of polyhedral sets $\{P_t\}_{t \in\mathbb{R}}$ is called a
	\emph{$C^1$-perturbation} of $P$ if
	\[
		P_t \coloneqq \bigcap_{k\in \mathcal{K}}
		\left\{x \in \mathbb{R}^d \mid n_k(t) \cdot x < a_k(t) \right\}
		\quad \forall t \in \mathbb{R}
		\quad\text{and}\quad
		P_0 = P,
	\]
	where $n_k\colon\mathbb{R}\to\mathbb{S}^{d-1}$ and
	$a_k\colon\mathbb{R}\to(0,\infty)$ are $C^1$-functions, 
	for $k\in\mathcal{K}$.
	Observe that the $n_k$ and $a_k$ functions represent normal vectors and
	positive distances to the origin, respectively.
	Thus, a \mbox{$C^1$-perturbation} is determined by its function pairs
	$\{(n_{k},a_{k})\}_{k\in\mathcal{K}}$.
	Moreover, the $|\mathcal{K}|$ facets of $P$ can be expressed as
	$F_k\cap\overline{P}$ for $k\in\mathcal{K}$, where $F_k$ is the hyperplane
	\begin{equation}\label{eq:fc_def}
		F_k = \left\{x\in\mathbb{R}^d\mid n_k(0) \cdot x = a_k(0)\right\}.
	\end{equation}
\end{definition}

\begin{definition}[see \cite{MR2039968,MR2336001}]
	Let $P\in\mathcal{P}$ be a simple polytope.
	Let $\mathcal{V}$ be the set of vertices of $P$.
	There exist continuous functions $b_v\colon P\to\mathbb{R}$,
	for each $v\in\mathcal{V}$,
	called \emph{barycentric coordinates} of $P$,
	which satisfy the following properties: given $x\in P$,
	\begin{itemize}
		\item[(i)] $b_v(x)>0$ for all $v\in\mathcal{V}$ (positivity),
		\item[(ii)] $\displaystyle\sum_{v\in\mathcal{V}}b_v(x)=1$
		(partition of unity), and
		\item[(iii)] $\displaystyle\sum_{v\in\mathcal{V}}b_v(x)v=x$
		(linear precision).
	\end{itemize}
	Let $\text{ind}(v)$ denote the set of indices $k\in\mathcal{K}$
	such that the facet $F_k\cap\overline{P}$ contains the vertex $v$. 
	Since $P$ is simple, we have $|\text{ind}(v)|=d$.
	Each barycentric coordinate $b_v$ is given by
	\[
		b_v(x)\coloneqq\frac{\omega_v(x)}{
			\displaystyle \sum_{v'\in\mathcal{V}}\omega_{v'}(x)
		},
		\quad\text{with}\quad
		\omega_v(x)\coloneqq\frac{|\det([n^\top_k]_{k\in\text{ind}(v)})|}{
			\displaystyle \prod_{k\in\text{ind}(v)} n_k\cdot(v-x) 
		},
	\]
	where $n_k$ is the unit normal vector to $F_k$,
	and $[n^\top_k]_{k\in\text{ind}(v)}$ is the square matrix with rows $n_k$,
	for $k\in\text{ind}(v)$.
	It was proved in \cite{MR3166966} that the barycentric coordinates
	have bounded gradient on $P$, namely, there exists a constant
	$C_{P,d}>0$ such that
	\begin{equation}\label{estimate-nabla-bv}
		\sup_{x\in P}\sum_{v\in\mathcal{V}}|\nabla b_v(x)|\leq C_{P,d}.
	\end{equation}
	This, the mean value theorem, and the fact that $P$ is convex imply that
	each $b_v$ is Lipschitz on $P$, and therefore, it can be uniquely
	extended to a Lipschitz function on $\overline{P}$.
	This extension preserves the partition of unity and linear precision
	properties, but it can vanish on some faces.
\end{definition}

The following result was proved in \cite[Thm.~13]{MR5011129}.
Here, we write it for reference, but in terms of \mbox{$C^1$-perturbations}.

\begin{theorem}\label{thm:Tbi}
	Let $P\in\mathcal{P}$ be a simple polytope with a
	$C^1$-perturbation $\{P_t\}_{t\in\mathbb{R}}$.
	Then there exist $\tau>0$ and a map
	$T\colon\overline{P}\times(-\tau,\tau)\to\mathbb{R}^{d}$
	such that $T(P,t)=P_t$, $T(\partial P,t)=\partial P_t$, and
	$T(\cdot,t)\colon\overline{P}\to\overline{P_t}$ is a
	bi-Lipschitz homeomorphism, for each $t\in(-\tau,\tau)$. Furthermore,
	\begin{equation}\label{eq:bdexpr}
		\int_{\partial P}\phi\,\theta\cdot \nu=
		\sum_{k\in\mathcal{K}}
		\left(
		a_{k}^{\prime}(0)\int_{F_k\cap\overline{P}}\phi-
		\begin{bmatrix}
			\displaystyle\int_{F_k\cap\overline{P}}\phi\,x_{1} \\
			\vdots \\
			\displaystyle\int_{F_k\cap\overline{P}}\phi\,x_{d}
		\end{bmatrix} \cdot n_{k}^{\prime}(0)
		\right)
		\quad \forall\phi\in L^{2}(\partial P),
	\end{equation}
	where $\theta\colon\overline{P}\to\mathbb{R}^d$ is defined by
	$\theta(x)\coloneqq\partial_{t}T(x,0)$,
	$\{(n_{k},a_{k})\}_{k\in\mathcal{K}}$ are the function pairs that determine
	$\{P_t\}_{t\in\mathbb{R}}$ in Definition~\ref{def3.3}, and
	$F_k\cap\overline{P}$ is the facet of $P$
	associated to $\{n_k,a_k\}$.
\end{theorem}

\begin{remark}\label{rem:Tmap}
	According to \cite[Thm.~13]{MR5011129},
	the map $T\colon\overline{P}\times(-\tau,\tau)\to\mathbb{R}^{d}$
	has the form
	\[
		T(x,t)=\sum_{v\in\mathcal{V}}b_{v}(x)z_{v}(t),
	\]
	where each $z_v:(-\tau,\tau)\to\mathbb{R}^d$ is a \mbox{$C^1$-function}
	such that $z_v(0)=v$ and
	\begin{equation}\label{eq:zv}
		\sum_{v\in\mathcal{V}}|z_v(t)-v| < \frac{1}{2C_{P,d}}
		\quad\forall t\in(-\tau,\tau).
	\end{equation}
	It is shown in \cite[Lem.~7]{MR5011129} that 
	the set of function pairs $\{(n_{k},a_{k})\}_{k\in\mathcal{K}}$
	uniquely determines the functions $\{z_v\}_{v\in\mathcal{V}}$ via
	the implicit function theorem.
	Moreover, $T$ can be expressed as a perturbation of the identity, namely
	\[
		T(x,t)=x+\sum_{v\in\mathcal{V}}b_{v}(x)(z_{v}(t)-v),
	\]
	from which we obtain the expressions
	\[
		\theta(x)=\partial_{t}T(x,0)=\sum_{v\in\mathcal{V}}b_v(x)z^{\prime}_v(0)
	\]
	and
	\begin{equation}\label{eq:tht_Dtht}
		D\theta(x)=\sum_{v\in\mathcal{V}}z^{\prime}_v(0)\otimes\nabla b_v(x).
	\end{equation}
\end{remark}

The proposition below will allow us to employ the domain perturbation induced
by a $C^1$-perturbation within the framework of Theorem~\ref{thm:sp}.

\begin{proposition}\label{prop:ext}
	There exists an extension
	$\widetilde{T}\colon\mathbb{R}^d\times(-\tau,\tau)\to\mathbb{R}^d$
	of $T$ obtained in Theorem~\ref{thm:Tbi}, that is,
	\[
		\widetilde{T}(x,t)=T(x,t)
		\quad\forall(x,t)\in\overline{P}\times(-\tau,\tau),
	\]
	such that $\widetilde{T}(\cdot,0)=\mathrm{Id}$,
	$ \widetilde{T}(\cdot,t)-\mathrm{Id}
	\in W^{1,\infty}(\mathbb{R}^d;\mathbb{R}^d)$,
	and $t\mapsto\widetilde{T}(\cdot,t)$ is differentiable.
\end{proposition}

\begin{proof}
	Applying Kirszbraun's theorem \cite{Kirszbraun},
	we can extend the barycentric coordinates of $P$ to Lipschitz functions
	$b^{K}_v\colon\mathbb{R}^d\to\mathbb{R}$,
	for each $v\in\mathcal{V}$.
	Let $p_{[0,1]}\colon\mathbb{R}\to [0,1]$ denote the metric projection
	onto the interval $[0,1]$.
	Let $\widetilde{b}_v \coloneqq p_{[0,1]} \circ b^{K}_v$.
	Given $x \in \overline{P}$, we have
	\[
		\widetilde{b}_v (x) = p_{[0,1]} (b^{K}_v (x))
		= p_{[0,1]} (b_v(x)) = b_v(x)
	\]
	since $b^{K}_v$ extends $b_v$ and $0\leq b_v(x) \leq1$,
	and for $x, y \in \mathbb{R}^d$ it holds that
	\[
		|\widetilde{b}_v(x) - \widetilde{b}_v(y)| \leq
		|p_{[0,1]} (b^{K}_v (x)) - p_{[0,1]} (b^{K}_v (y))| \leq
		|b^{K}_v (x) - b^{K}_v (y)|
	\]
	since $p_{[0,1]}$ is Lipschitz.
	Moreover, $0 \leq \widetilde{b}_v(x) \leq 1$ for all $x\in \mathbb{R}^d$.
	It follows that $\widetilde{b}_v$ is a bounded Lipschitz extension of $b_v$.
	Using these extensions, we construct the map
	$\widetilde{T}\colon\mathbb{R}^d\times(-\tau,\tau)\to\mathbb{R}^d$
	defined by
	\[
		\widetilde{T}(x,t)\coloneqq
		x+\sum_{v\in\mathcal{V}}\widetilde{b}_v(x)(z_v(t)-v).
	\]
	Clearly, $\widetilde{T}$ extends $T$.
	Since $z_v(0)=v$ for each $v\in\mathcal{V}$, we have
	$\widetilde{T}(\cdot,0)=\mathrm{Id}$.
	From \eqref{eq:zv},
	\[
		|\widetilde{T}(x,t) - x| \leq
		\sum_{v\in\mathcal{V}} |\widetilde{b}_v(x)| |z_v(t)-v|,
	\]
	and
	\[
		|\widetilde{T}(x,t) - x - (\widetilde{T}(y,t) - y)| \leq
		\sum_{v\in\mathcal{V}}
		|\widetilde{b}_v(x)-\widetilde{b}_v(y)||z_v(t)-v|,
	\]
	it follows that
	$\widetilde{T}(\cdot,t) - \mathrm{Id}$ is bounded and Lipschitz, and hence
	it belongs to $W^{1,\infty}(\mathbb{R}^d;\mathbb{R}^d)$.
	The differentiability is a consequence of the fact that each
	$z_v$ is a $C^1$-function.
\end{proof}

\newpage

\begin{remark}
	While the bi-Lipschitz property of $T(\cdot, t)$ is needed to justify a
	change of variables 
	(as required, for instance, in the computation of shape derivatives),
	the boundary expression \eqref{eq:bdexpr}
	can be directly obtained from the requirement that $T$ must preserve
	the facets of the polyhedral domain.
	Indeed, given a facet $F_k\cap\overline{P}$, if
	\[
		T(x,t) \cdot n_k(t) = a_k(t)
		\quad
		\forall x\in F_k\cap\overline{P},\,\forall t\in(-\tau,\tau),
	\]
	then, differentiating with respect to $t$ and evaluating at $t=0$, we obtain
	\[
	\begin{aligned}
		\partial_t T(x,0)\cdot n_k(0) + T(x,0)\cdot n_k^{\prime}(0) &= a_k^{\prime}(0)\\
		\theta(x)\cdot n_k(0) + x\cdot n_k^{\prime}(0) &= a_k^{\prime}(0)\\
		\theta(x)\cdot \nu &= a_k^{\prime}(0) - x\cdot n_k^{\prime}(0)\quad\forall x\in F_k\cap\overline{P}
	\end{aligned}
	\]
	since $n_k(0)=\nu$ on $F_k\cap\overline{P}$.
	Integrating over all the facets yields \eqref{eq:bdexpr}.
\end{remark}

\section{Optimization problem}\label{sec:opt_prb}

In this section, we present the optimization problem that we solve numerically.
We also show that this problem is equivalent to
the polyhedral Rayleigh--Faber--Krahn inequality.

Consider the shape optimization problem
\begin{equation}\label{eq:min_poly}
	\min_{P\in\mathcal{P}}\lambda(P)
	\quad\text{subject to}\quad |P|=\pi,
\end{equation}
and the Rayleigh--Faber--Krahn inequality for polyhedral domains
\begin{equation}\label{eq:fki_poly}
	P^*\in\mathcal{P}:\quad
	|P^*|^{2/d}\lambda(P^*)
	\leq|P|^{2/d}\lambda(P)
	\quad \forall P\in\mathcal{P}.
\end{equation}
Recall that $\mathcal{P}$ is the class of polytopes
with $|\mathcal{K}|$ facets and dimension $d$.
Moreover, we have assumed that
each polytope in $\mathcal{P}$ is open,
contains the origin, and that its representation as
a polyhedral set is irreducible.

\begin{lemma}\label{lem:equiv}
	Suppose that problems \eqref{eq:min_poly} and \eqref{eq:fki_poly}
	admit solutions.
	Then the two problems are equivalent,
	in the sense that their solutions coincide up to a scaling.
\end{lemma}
\begin{proof}
	Let $P^{\min}$ be a minimizer of \eqref{eq:min_poly}.
	For $P\in\mathcal{P}$, define the scaled copy
	\[
		P_{\pi}\coloneqq(\pi^{1/d}|P|^{-1/d})P.
	\]
	Then $|P_{\pi}|= (\pi^{1/d}|P|^{-1/d})^{d}|P|=\pi$.
	The minimality of $P^{\min}$ yields
	\[
		\lambda(P^{\min})\leq\lambda(P_{\pi})
		=(\pi^{1/d}|P|^{-1/d})^{-2}\lambda(P),
	\]
	and hence
	\[
		|\pi|^{2/d}\lambda(P^{\min})\leq
		|P|^{2/d}\lambda(P),
	\]
	which proves that $P^{\min}$ is a solution to \eqref{eq:fki_poly}.

	\noindent Conversely, let $P^*$ be a solution to \eqref{eq:fki_poly}.
	For any $P\in\mathcal{P}$ with $|P|=\pi$ it holds that
	\[
		|P^*|^{2/d}\lambda(P^*) \leq \pi^{2/d}\lambda(P),
	\]
	which implies
	\[
		\lambda(P)\geq
		(\pi^{1/d}|P^*|^{-1/d})^{-2}\lambda(P^*)=
		\lambda((\pi^{1/d}|P^*|^{-1/d})P^*).
	\]
	Therefore, $P^*_\pi$ is a minimizer of \eqref{eq:min_poly}.
\end{proof}

We now reformulate \eqref{eq:min_poly} as a
finite-dimensional optimization problem
by using polyhedral sets.

Let $a\colon\mathbb{R}\to(0,\infty)$ be a bijective $C^1$-function.
Let $n\colon\mathbb{R}^{d-1}\to\mathbb{S}^{d-1}$ be
the hyperspherical coordinate map of $\mathbb{S}^{d-1}$, that is,
\begin{equation}\label{eq:hyp_coor}
	n(\gamma)=
	\begin{pmatrix}
		\cos(\gamma^1) \\
		\sin(\gamma^1)\cos(\gamma^2) \\
		\sin(\gamma^1)\sin(\gamma^2)\cos(\gamma^3) \\
		\vdots \\
		\sin(\gamma^1)\sin(\gamma^2)\cdots\sin(\gamma^{d-2})\cos(\gamma^{d-1})\\
		\sin(\gamma^1)\sin(\gamma^2)\cdots\sin(\gamma^{d-2})\sin(\gamma^{d-1})
	\end{pmatrix}
\end{equation}
with $\gamma=(\gamma^1,\ldots,\gamma^{d-1})$.
Define the parameter space
\begin{equation}\label{eq:Gr}
	\Gamma\coloneqq
	\left\{r\in\mathbb{R}^{d|\mathcal{K}|}\mid P(r)\in\mathcal{P}\right\},
\end{equation}
where the elements of $\Gamma$ have the form
\[
	r=(\gamma_1,\mu_1,\ldots,\gamma_{|\mathcal{K}|},\mu_{|\mathcal{K}|})
\]
with $\gamma_k=(\gamma_k^1,\ldots,\gamma_k^{d-1})\in\mathbb{R}^{d-1}$ 
and
$\mu_k\in\mathbb{R}$ for each $k\in\mathcal{K}$,
and $P(r)$ is the polyhedral set given by
\begin{equation}\label{eq:Pr}
	P(r)\coloneqq
	\bigcap_{k\in\mathcal{K}}
	\left\{x\in\mathbb{R}^d \mid n(\gamma_k)\cdot x < a(\mu_k)\right\}.
\end{equation}
Here, we have assumed $\mathcal{K}=\{1,2,\ldots,|\mathcal{K}|\}$. 
Each $\gamma_k$ is a vector of angles corresponding to the normal vector $n(\gamma_k)$.
Since $a$ is bijective, the distances to the origin $a(\mu_k)$ are parametrized over $\mathbb{R}$.
In this way, a polyhedral set can be represented by its coordinates in the parameter space $\Gamma$.

With this construction, we write \eqref{eq:min_poly} as
\begin{equation}\label{eq:min_par}
	\min_{r\in\Gamma} \lambda(P(r))\quad\text{subject to}\quad |P(r)|=\pi,
\end{equation}
which is the minimization problem that we solve numerically for $d=3$.

\section{Numerical method}\label{sec:num_mtd}

We apply the Lagrangian method \cite{MR4589221} to solve \eqref{eq:min_par}.
Let $\mathscr{L}\colon\Gamma\times\mathbb{R}^3\to\mathbb{R}$
be the Lagrangian functional defined by
\[
	\mathscr{L}(r,s,l,m)\coloneqq
	\lambda(P(r))+(l,m)\cdot(|P(r)|-\pi-s,s)+
	\frac{p}{2}s^2-\frac{q}{2}(l-m)^2,
\]
where $P(r)$ is given by \eqref{eq:Pr}, $l$ and $m$ are Lagrange multipliers,
$s$ is a slack variable, and $p,q>0$ are penalty parameters.
The algorithm proposed in \cite{MR4589221} applied to \eqref{eq:min_par}
takes the form of Algorithm~\ref{Lag-algo}.

\begin{algorithm}
	\caption{Lagrangian method}\label{Lag-algo}
	\refstepcounter{equation}\label{eq:algo}
	\KwIn{
		$\eta>0$ sufficiently large, $\sigma\in(0,1)$,
		$p\in(1,\infty)$, $q\in(0,1)$.
	}
	Set $(r^{0},s^{0},l^{0},m^{0})$ and $g^0\in(0,1]$.\\
	\For{$i = 0,1,2,\ldots$}{
		$r^{i+1}=r^{i}-\dfrac{1}{\eta}\nabla_{r}\mathscr{L}(r^i,s^i,l^i,m^i)$\\
		$m^{i+1}=m^i+g^i\dfrac{l^i-m^i}{(l^{i}-m^{i})^{2}+1}$
		\hfill (\theequation)\\
		$l^{i+1}=m^{i+1}+\dfrac{p}{1+pq}\left(|P(r^{i+1})|-\pi\right)$\\
		$s^{i+1}=\dfrac{1}{p}(l^{i+1}-m^{i+1})$\\
		$g^{i+1}=\sigma g^{i}$
	}
\end{algorithm}

Algorithm~\ref{Lag-algo} generates a sequence
$(r^{i},s^{i},l^{i},m^{i})_{i\in\mathbb{N}}$ which is expected to converge
to a limit point $(\bar{r},\bar{s},\bar{l},\bar{m})$, in which
$(\bar{r}, \bar{l})$ is a KKT point of problem \eqref{eq:min_par},
$\bar{s}=0$, and $\bar{l}=\bar{m}$.
Thus, a sequence of polyhedral domains
$\left(P(r^i)\right)_{i\in\mathbb{N}}$ is obtained,
which is expected to converge, at least numerically,
to a solution of \eqref{eq:polyFK}.
In this work, we restrict ourselves to a formal setting and
do not address the verification of the assumptions in \cite{MR4589221}
required for the convergence analysis of Algorithm~\ref{Lag-algo}.

\subsection{Gradient computation}

In order to implement Algorithm~\ref{Lag-algo}, we compute
the gradient of the Lagrangian functional $\mathscr{L}$ with respect to $r$.
This is precisely where the shape derivatives of
first Dirichlet eigenvalue of the Laplacian and the volume,
together with the boundary expression \eqref{eq:bdexpr}, are used.
Let
\[
	\begin{aligned}
	r &=
	\left(
		\gamma_1,\mu_1,\ldots,\gamma_{|\mathcal{K}|},\mu_{|\mathcal{K}|}
	\right)\in\Gamma,\\
	\delta r &=
	\left(
		\delta\gamma_1,\delta\mu_1,\ldots,
		\delta\gamma_{|\mathcal{K}|},\delta\mu_{|\mathcal{K}|}
	\right)\in\mathbb{R}^{d|\mathcal{K}|},
	\end{aligned}
\]
where $\Gamma$ is the set of coordinates \eqref{eq:Gr}.
The directional derivative of $\mathscr{L}$ at $r$ in the direction $\delta r$
is given by
\begin{equation}
	\begin{aligned}
		\nabla_r\mathscr{L}(r,s,l,m)\cdot{\delta r}
		&=\left.\frac{d}{dt}\mathscr{L}(r+t{\delta r},s,l,m)\right|_{t=0}\\
		&=\left.\frac{d}{dt}\lambda(P(r+t{\delta r}))\right|_{t=0}
		+l\left.\frac{d}{dt}|P(r+t{\delta r})|\right|_{t=0}.
	\end{aligned}	
\end{equation}
For each $k\in\mathcal{K}$, define the functions
$n_k\colon\mathbb{R}\to\mathbb{S}^{d-1}$ and
$a_k\colon\mathbb{R}\to(0,\infty)$ by
\[
	n_k(t)\coloneqq n(\gamma_k+t\delta\gamma_k)
	\quad\text{and}\quad
	a_k(t)\coloneqq a(\mu_k+t\delta\mu_k),
\]
where $n$ is the hyperspherical coordinate map \eqref{eq:hyp_coor} and
$a$ is the bijective $C^1$-function used in the construction of $\Gamma$.
Note that $\{(n_k,a_k)\}_{k\in\mathcal{K}}$ determines a
\mbox{$C^1$-perturbation} $\{P_t\}_{t\in\mathbb{R}}$ of $P(r)$.
It follows from Definition~\ref{def3.3} and \eqref{eq:Pr} that
\[
	P_t = P(r+t\delta r)\quad\text{for }t\in\mathbb{R},
\]
and then
\begin{equation}\label{eq:L1}
	\nabla_r\mathscr{L}(r,s,l,m)\cdot{\delta r}=
	\left.\frac{d}{dt}\lambda(P_t)\right|_{t=0}
	+l\left.\frac{d}{dt}|P_t|\right|_{t=0}.	
\end{equation}
In general, for $t\neq0$, the set $P(r+t\delta r)$
may not belong to $\mathcal{P}$:
it may be unbounded or its representation as a polyhedral set may be reducible.
However, it was proved in \cite[Lem.~6]{MR5011129} that
for $t$ sufficiently small, \mbox{$C^1$-perturbations} of
irreducible bounded polyhedral sets remain irreducible and bounded.
Thus, for $t$ sufficiently small, $P(r+t\delta r)\in\mathcal{P}$.

We now introduce an important assumption that allows us to carry out the subsequent analysis.
\begin{assumption}\label{assump:simple}
	$P(r)$ is simple, i.e.,
	each vertex of $P(r)$ belongs to exactly $d$ facets.
\end{assumption}

Under Assumption~\ref{assump:simple}, we can apply Theorem~\ref{thm:Tbi}
to the polyhedral set $P(r)$ and the \mbox{$C^1$-perturbation} of $P(r)$ defined above:
there exist $\tau>0$ and a map
$T:\overline{P(r)}\times(-\tau,\tau)\to\mathbb{R}^{d}$
such that
\[
	T(P(r),t)=P_t \quad \text{for } t\in(-\tau,\tau).
\]
We now define $\Psi_t\coloneqq \widetilde{T}(\cdot,t)$ for $t\in[0,\tau)$,
where $\widetilde{T}$ is the extension of $T$ obtained in Proposition~\ref{prop:ext}.
Thus, we have
\[
	P(r+t\delta r) = P_t = T(P(r),t) = \Psi_t(P(r)).
\]
Since $\{\Psi_t\}_{t\in[0,\tau)}$ satisfies the assumptions of Theorem~\ref{thm:sp},
it follows from \eqref{eq:L1}, \eqref{eq:sp_vol}, and \eqref{eq:sp_eig} that
\[
	\nabla_r\mathscr{L}(r,s,l,m)\cdot{\delta r}
	=\int_{\partial P(r)}
	\left(l-(\partial_{\nu}u_{P(r)})^{2}\right)\theta\cdot\nu,
\]
where $u_{P(r)}$ is the $L^2$-normalized
eigenfunction associated to $\lambda(P(r))$,
$\theta\coloneqq\partial_{t}T(\cdot,0)$,
and $\nu$ is the outward unit normal vector to $\partial P(r)$.
Finally, employing the boundary expression \eqref{eq:bdexpr} with
$\phi=l-(\partial_{\nu}u_{P(r)})^{2}$ we obtain
\[
	\nabla_r\mathscr{L}(r,s,l,m)\cdot{\delta r}
	=\left(
		\begin{array}{c}
		\mathcal{G}_1 \\ 
		\mathcal{M}_1 \\
		\vdots \\
		\mathcal{G}_{|\mathcal{K}|} \\  
		\mathcal{M}_{|\mathcal{K}|}
		\end{array}
		\right) \cdot \delta r,
\]
where $\mathcal{G}_k\in\mathbb{R}^{d-1}$ and $\mathcal{M}_k\in\mathbb{R}$ are given by
\[
	\mathcal{G}_k=-\left(Dn(\gamma_{k})\right)^\top
	\begin{pmatrix}
		\displaystyle\int_{F_k\cap\overline{P(r)}}\left(l-(\partial_{\nu}u_{P(r)})^{2}\right)\,x_{1}\\
		\vdots\\
		\displaystyle\int_{F_k\cap\overline{P(r)}}\left(l-(\partial_{\nu}u_{P(r)})^{2}\right)\,x_{d}
	\end{pmatrix}
\]
and
\[
	\mathcal{M}_k= a'(\mu_k)\int_{F_k\cap\overline{P(r)}}\left(l-(\partial_{\nu}u_{P(r)})^{2}\right).
\]
Here, $Dn(\gamma_k)$ is the Jacobian matrix of the hyperspherical
coordinate map $n$ evaluated at $\gamma_k$
and $\{F_k\}_{k\in\mathcal{K}}$ are the hyperplanes \eqref{eq:fc_def}.
Since $\delta r$ was arbitrary,
we have obtained the expression of the gradient of $\mathscr{L}$ with respect to $r$.

In summary, prior to each iteration $i$ of Algorithm~\ref{Lag-algo}, we
\begin{enumerate}[label=(\Roman*)]
	\item construct the domain $P(r^{i})$ (with $r^i$ obtained in the previous iteration);
	\item compute the eigenvalue $\lambda(P(r^{i}))$ and
	the eigenfunction $u_{P(r^{i})}$ by solving \eqref{eq:eig_pb} with $\Omega= P(r^{i})$;
	\item obtain the gradient $\nabla_{r}\mathscr{L}(r^{i},s^{i},l^{i},m^{i})$ computing
	$\mathcal{G}_k$ and $\mathcal{M}_k$ for $k\in\mathcal{K}$.
\end{enumerate}
In step (I), it is assumed that $P(r^{i})\in\mathcal{P}$ and
that, according to Assumption \ref{assump:simple}, it is simple.

\section{Numerical results}\label{sec:num}

We perform numerical tests to solve the minimization problem \eqref{eq:min_par}
with $\Gamma$ parametrizing polytopes with $|\mathcal{K}|$ facets,
as shown in Section \ref{sec:num_mtd}.
We consider the cases $|\mathcal{K}|= 5,\ldots,14$.

The code was written in \texttt{Python}.
The eigenvalue problem \eqref{eq:eig_pb} was solved using
\texttt{FEniCSx~0.9}~\cite{fenicsx} and the \texttt{SLEPc} library via its
\texttt{Python} interface \texttt{slepc4py}~\cite{Hernandez:2005:SSF}. 
The polyhedral meshes were constructed using 
\texttt{Gmsh}~\cite{geuzaine2009gmsh}, employing tetrahedral finite elements.
For simplicity, no stopping criterion was implemented.
Instead, each numerical experiment was run for a fixed number
of $1500$ iterations, which proved sufficient to obtain satisfactory results in all cases.
This produced the sequence of polytopes
\[
	P^i \coloneqq P(r^i),\quad r^i\in\Gamma,
	\quad\text{for }i=0,\ldots,1500.
\]
We denote the last iterate by $P^* \coloneqq P(r^{1500})$.
The effectiveness of this strategy is illustrated by the monotone decrease of
$\lambda(P^i)|P^i|^{2/3}$ throughout the iterations in all cases.

The parameter $\eta$ in Algorithm~\ref{Lag-algo} starts with $\eta=25$ and
is increased following the rule $\eta=\eta+25$ at the iterations $i=500$, $900$, $1200$, and $1400$.
For computational efficiency, we also adopt a mesh refinement strategy, where
the number of finite elements is increased
by reducing the mesh parameter in the \texttt{Gmsh} configuration at these same iterations $i=500$, $900$, $1200$, and $1400$.
This results in jumps observed in the convergence history, see for instance  Figure~\ref{fig:p04}.
The remaining parameters in Algorithm~\ref{Lag-algo} take the following fixed values,
which are the same as in \cite{MR4589221}:
$(g^0,\sigma) = (0.5, 0.999)$, $(p,q) = (2000,0.5)$, and $s^0=l^0=m^0=0$.
The initial guesses $P(r^0)$ were truncated tetrahedra and cubes,
from which $r^0$ was obtained by inverting
the hyperspherical coordinate map \eqref{eq:hyp_coor}
and the distance to the origin function $a\colon\mathbb{R}\to(0,\infty)$,
which is chosen as
\[
	a(\mu)\coloneqq
	a_{\min}+(a_{\max}-a_{\min})\left(\frac{\exp\mu}{1+\exp\mu}\right),
\]
with $a_{\min} = 0.1$ and $a_{\max}=2.0$. This is a scaled sigmoid function
that avoids very small or large values
of the distances of the facets to the origin. Although it is not a bijective
function from $\mathbb{R}$ onto $(0,\infty)$, it works very well in practice because, under the given volume constraint, the distances
remain within the interval $(a_{\min}, a_{\max})$ during the iterations.

A summary of the numerical results is provided in Table~\ref{tab:results}.
In Figures~\ref{fig:p05}--\ref{fig:p14} are the three-dimensional polyhedra and
the values of $\lambda(P^i)|P^i|^{2/3}$ along the iterations,
for each test. The norm of $\nabla_{r}\mathscr{L}(r^i,s^i,l^i,m^i)$
and the constraint error $||P^i|-\pi|$ achieved values smaller that
$0.05$ and $0.005$, respectively, in all cases; see Figure~\ref{fig:gaps} for the results when $|\mathcal{K}|=14$.
The eigenvalues associated with the optimal shapes decrease as
the number of facets increases, while remaining strictly greater
than the eigenvalue of the ball, see Figure~\ref{fig:comparison}.
This monotonicity is similar to the case of polygons in two dimensions. 
Indeed, in \cite{MR2264470}  it was observed that
the first Dirichlet eigenvalues are monotonically decreasing with respect
to the number of polygon edges in numerical experiments and the monotonicity
was proved in  \cite{MR4742089} when the number of edges goes to infinity,
using an asymptotic expansion.

We now provide comments on the numerical solutions.
The tests with $|\mathcal{K}|=5$ and $|\mathcal{K}|=7$ facets yield prisms,
namely triangular and pentagonal prisms, respectively.
Platonic solids arise in the tests with $|\mathcal{K}|=4$, $|\mathcal{K}|=6$, and $|\mathcal{K}|=12$ facets,
visually resembling the tetrahedron, the cube, and the dodecahedron.

The optimal polyhedra with $|\mathcal{K}|=4, 5,6,7,8,9,10,12$ facets 
coincide with those obtained for the polyhedral version of the Saint--Venant
inequality in \cite{MR5011129},
where the maximization of the torsion rigidity
on polyhedral domains is investigated.
The test with $|\mathcal{K}|=14$ facets leads to
a particularly remarkable result:
the obtained polyhedron is isomorphic to the \emph{tetrakaidecahedron}
(also known as the truncated octahedron),
which appears in many contexts such as space-filling tessellations and
models of cellular structures \cite{d1992growth, patterns}.
To the best of our knowledge, the optimal polyhedral domains produced by
the tests with $|\mathcal{K}|=8, 9, 10, 11, 13$ facets
do not appear to correspond to any previously named polyhedra.

We observe that edge collapses and vertex mergers occur during the iterations
in most of the tests.
This process consists of a transition from one edge to another,
which does not correspond to a continuous transformation of the vertices.
Consequently, the polyhedra do not remain isomorphic throughout the iterations,
although they preserve their simple property.
This behavior was also reported in \cite{MR5011129}.

\begin{table}
\centering
\begin{tabular}{cccc}
\toprule
$\left|\mathcal{K}\right|$ & Max. FEs & $\lambda(P^{*})\,|P^{*}|^{2/3}$ & $P^{0}$\\
\midrule
4  & 266581 & 36.40775 & irregular tetrahedron \\
5  & 236770 & 32.69062 & tetrahedron with one vertex truncated \\
6  & 288441 & 29.68291 & tetrahedron with two vertices truncated \\
7  & 279845 & 28.85569 & tetrahedron with three vertices truncated \\
8  & 280577 & 28.20829 & cube with two vertices truncated \\
9  & 281825 & 27.67099 & cube with three vertices truncated \\
10 & 266290 & 27.33698 & cube with four vertices truncated \\
11 & 273031 & 27.12258 & cube with five vertices truncated \\
12 & 276151 & 26.83699 & cube with six vertices truncated \\
13 & 276047 & 26.74614 & cube with seven vertices truncated \\
14 & 282844 & 26.66541 & cube with all its vertices truncated \\
\bottomrule
\end{tabular}
\caption{
	Numerical results for different polyhedral configurations.
	The first column lists the number of facets.
	The second column shows the maximum number of finite elements used.
	The values of $\lambda(P^{*})\,|P^{*}|^{2/3}$ were
	obtained at the last iteration;
	recall that $\lambda(B)\,|B|^{2/3}=25.64635$,
	for the Euclidean ball $B$. The initial guesses $P^0$
	are given in the fourth column. 
}
\label{tab:results}
\end{table}

\begin{figure}
	\centering
	\includegraphics[width=\textwidth]{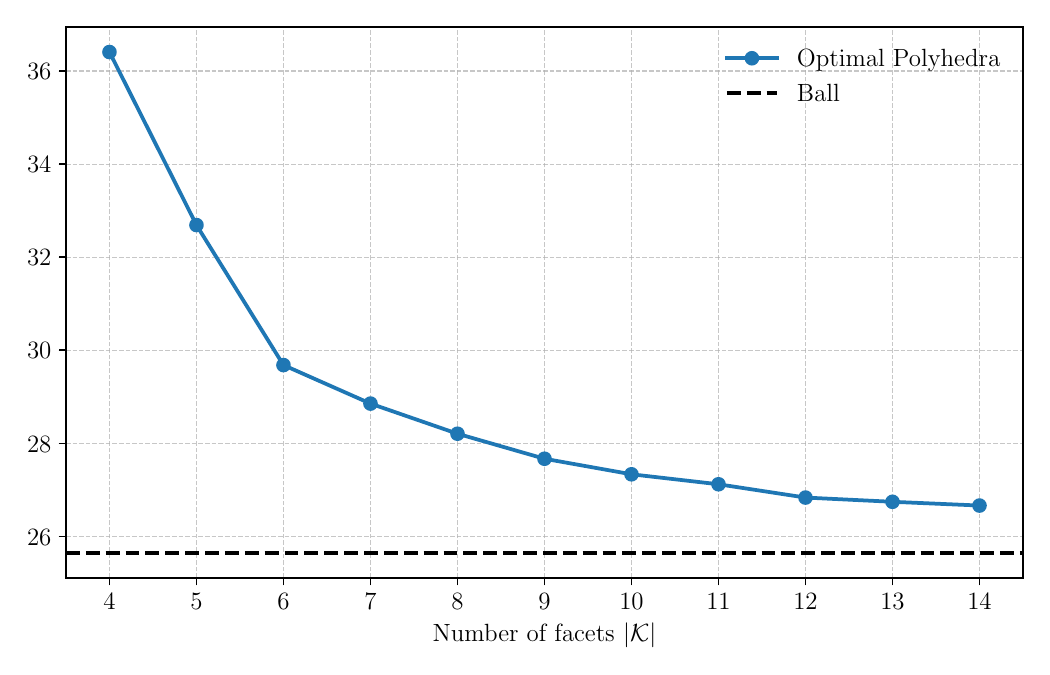}
	\caption{
		The values of $\lambda(P^{*})\,|P^{*}|^{2/3}$ 
		decrease as the number of facets increases, while remaining
		greater than the corresponding value of $\lambda(B)\,|B|^{2/3}$
		for the Euclidean ball. 
	}\label{fig:comparison}
\end{figure}

\section{Conclusion}

In this work, we have presented numerical results for the polyhedral
Rayleigh--Faber--Krahn inequality in three dimensions.
The shape optimization problem of minimizing the
first Dirichlet eigenvalue of the Laplacian over
polyhedral domains with a fixed number of facets was reformulated
as a finite-dimensional optimization problem,
and a first-order Lagrangian scheme was implemented
to compute optimal solutions.
Sensitivity analysis was carried out using
shape derivatives on polyhedral domains.

The results recover well-known polyhedra
(isomorphic to prisms, Platonic solids, and the tetrakaidecahedron),
as well as new non-regular polyhedra exhibiting symmetries
that are not evident a priori.
We conjecture that the minimizers of the Dirichlet eigenvalue
among polyhedra belong to the class of simple polytopes.

Possible directions for future research include the numerical investigation
of other spectral inequalities for polyhedra, such as the
Bossel--Daners and Szeg{\"o}--Weinberger inequalities,
corresponding to the Robin and Neumann problems, respectively.

\begin{figure}
	\centering
	\includegraphics[width=0.71\textwidth]{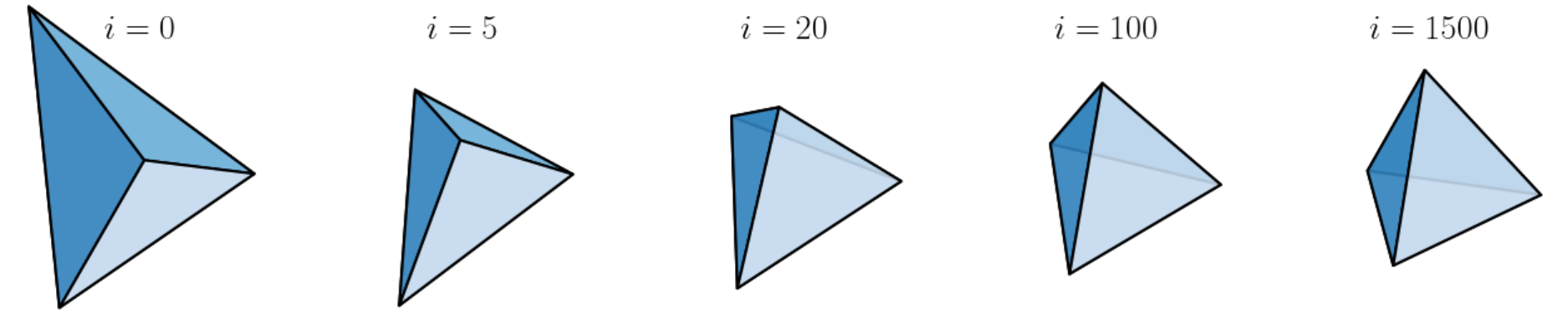}
	\vspace{0.1cm}
	\includegraphics[width=0.33\textwidth]{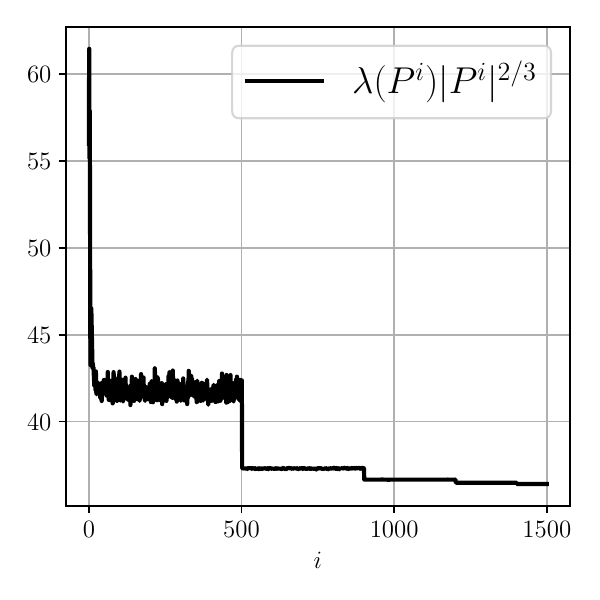}
	\hspace{0.2cm}
	\includegraphics[width=0.33\textwidth]{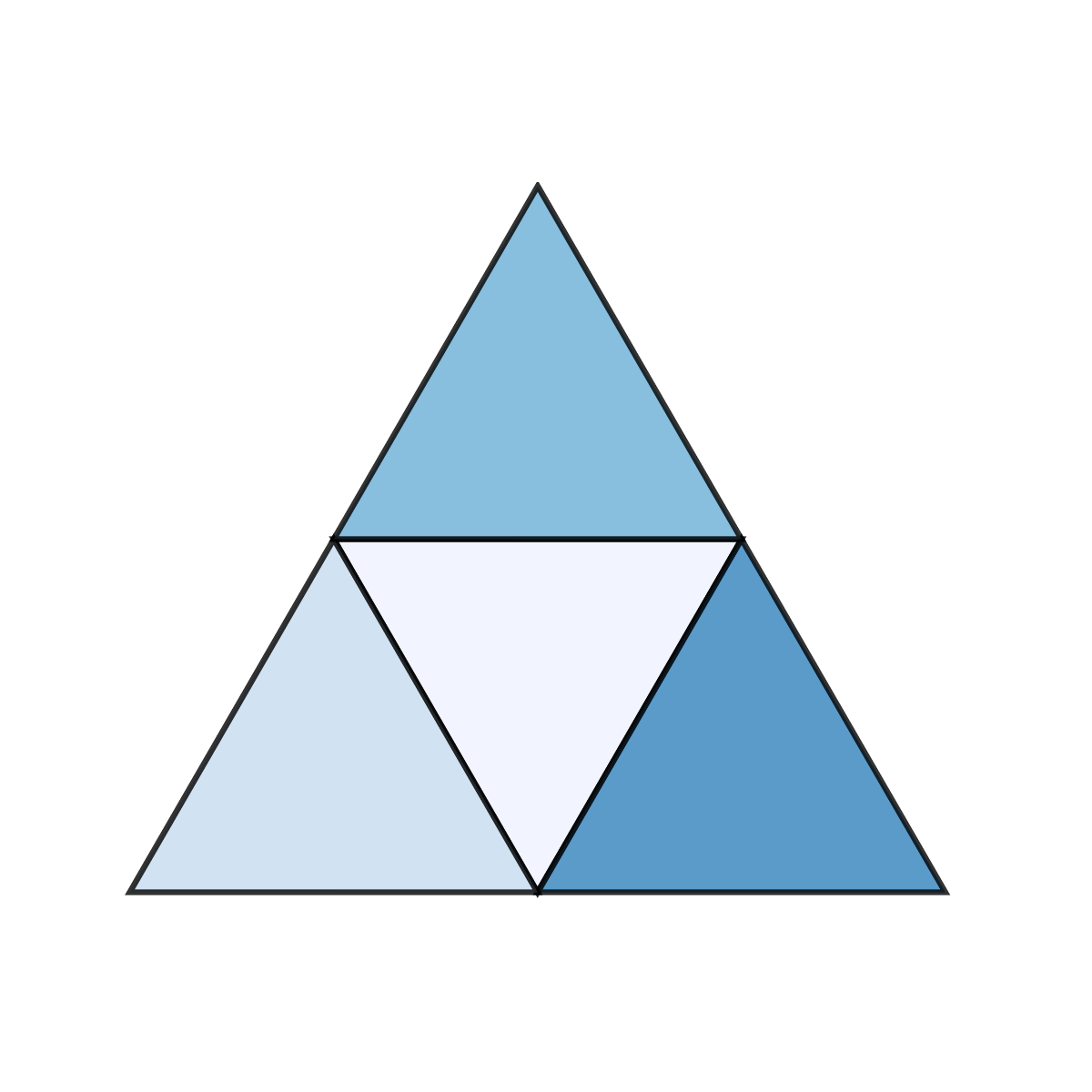}
	\caption{
		Case $\left|\mathcal{K}\right| = 4$.
		(Top) Polyhedra corresponding to iterations $i=0,5,20,100$, and $1500$. 
		(Bottom) Values of $\lambda(P^i)\,|P^i|^{2/3}$ and
		the unfolding of the optimal polyhedron $P^{*}$.
		The initial polyhedron $P^{0}$ was an irregular tetrahedron.
	}
	\label{fig:p04}
\end{figure}

\begin{figure}
	\centering
	\includegraphics[width=0.71\textwidth]{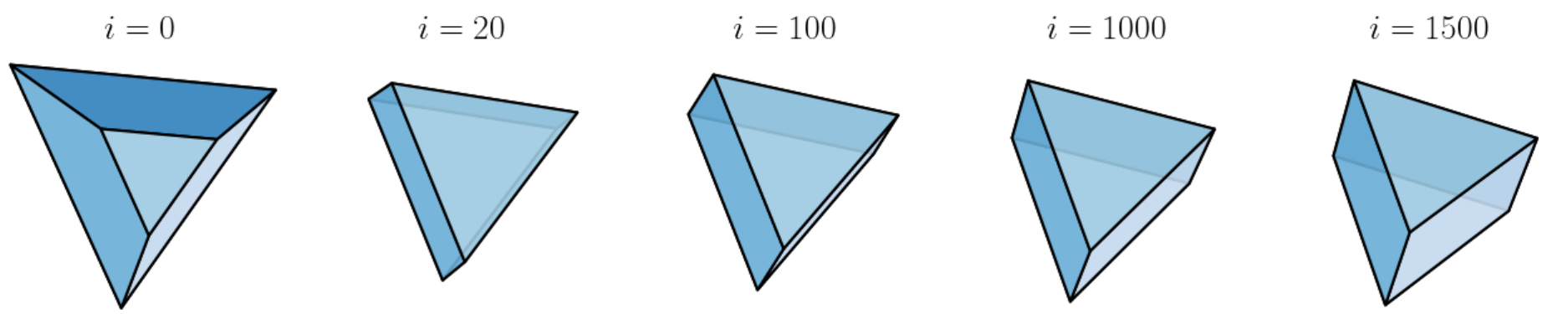}
	\vspace{0.1cm}
	\includegraphics[width=0.33\textwidth]{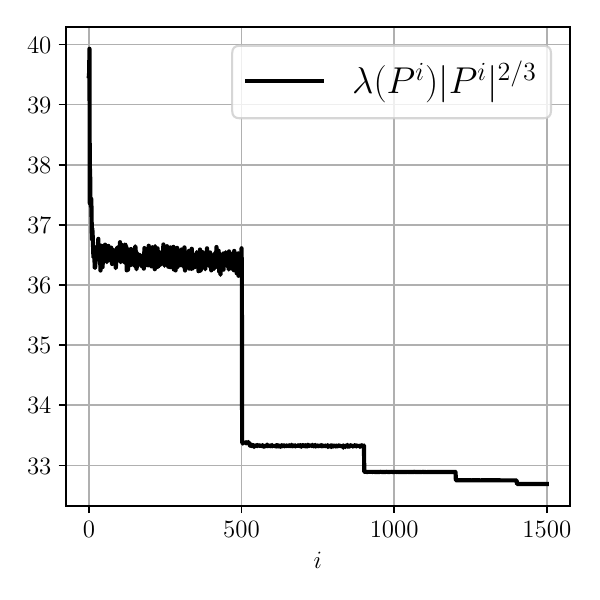}
	\hspace{0.2cm}
	\includegraphics[width=0.33\textwidth]{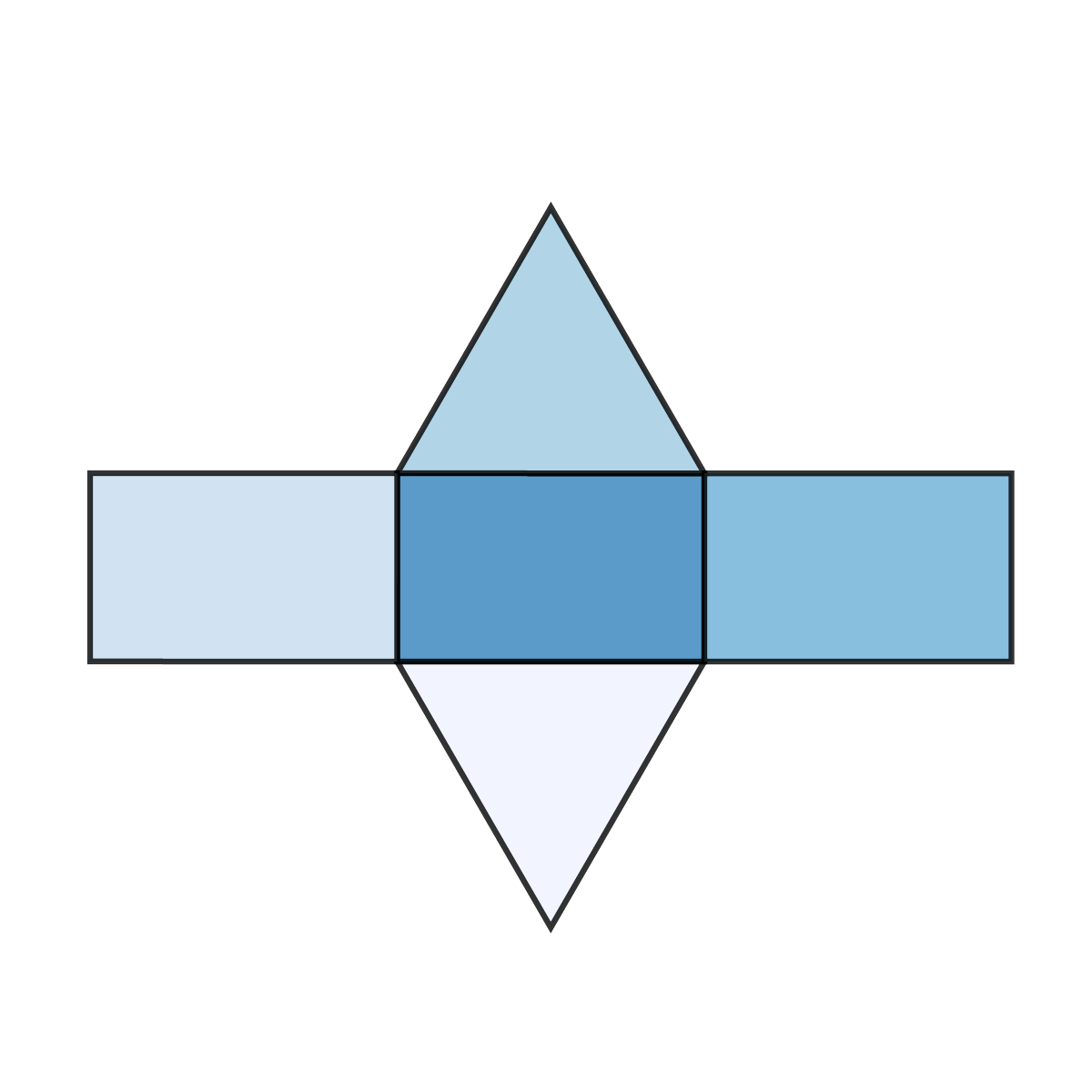}
	\caption{
		Case $\left|\mathcal{K}\right| = 5$.
		(Top) Polyhedra corresponding to iterations $i=0,20,100,1000$, and $1500$.
		(Bottom) Values of $\lambda(P^i)\,|P^i|^{2/3}$ and
		the unfolding of the optimal polyhedron $P^{*}$.
		The initial polyhedron $P^{0}$ was
		a tetrahedron with one truncated vertex.
		All polyhedra $P^{i}$ were isomorphic to a triangular prism.
	}
	\label{fig:p05}
\end{figure}

\begin{figure}
	\centering
	\includegraphics[width=0.71\textwidth]{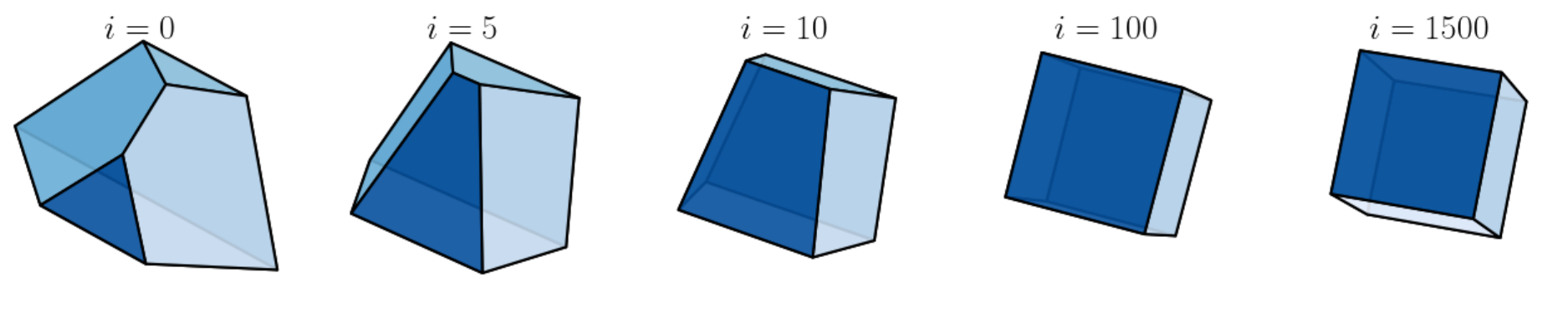}
	\vspace{0.1cm}
	\includegraphics[width=0.33\textwidth]{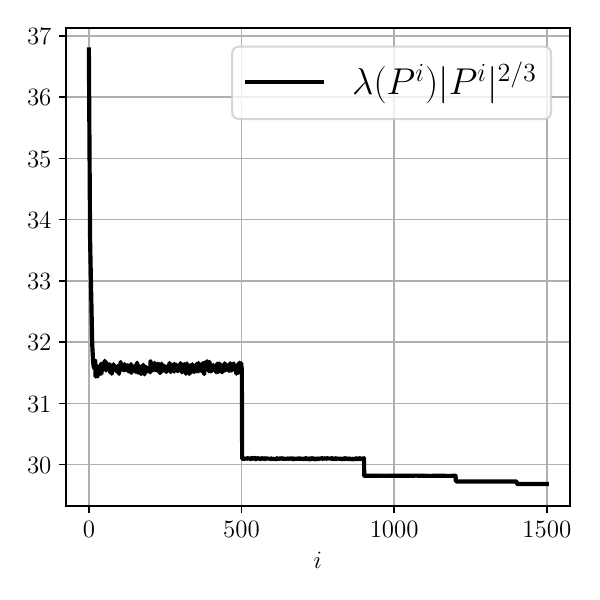}
	\hspace{0.2cm}
	\includegraphics[width=0.33\textwidth]{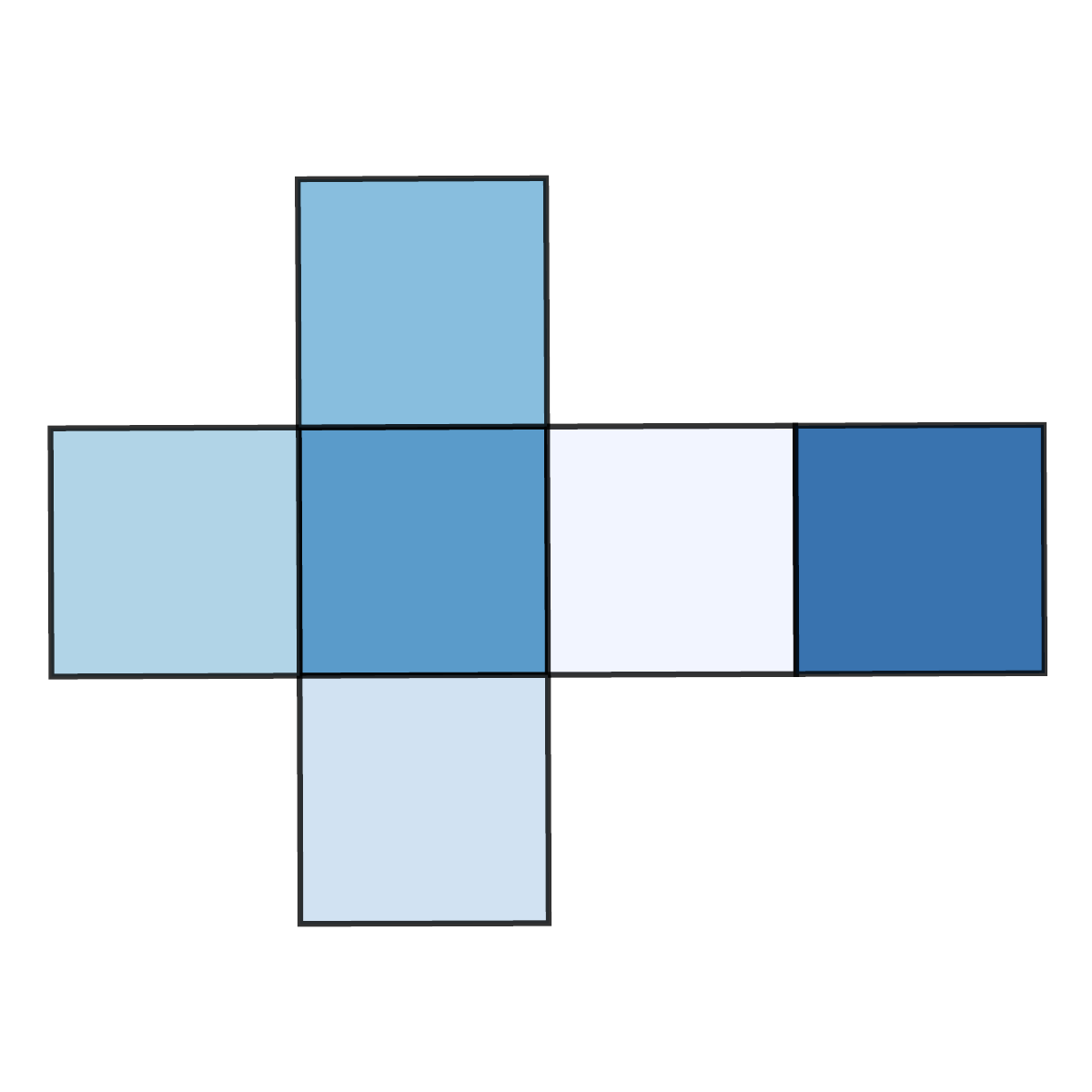}
	\caption{
		Case $\left|\mathcal{K}\right| = 6$.
		(Top) Polyhedra corresponding to iterations $i=0,5,10,100$, and $1500$.
		(Bottom) Values of $\lambda(P^i)\,|P^i|^{2/3}$ and
		the unfolding of the optimal polyhedron $P^{*}$.
		The initial polyhedron $P^{0}$ was
		a tetrahedron with two truncated vertices.
		All polyhedra $P^{i}$ for $i \geq 5$ were isomorphic to a cube.
		The initial and final polyhedra are not isomorphic.
	}
	\label{fig:p06}
\end{figure}

\begin{figure}
	\centering
	\includegraphics[width=0.71\textwidth]{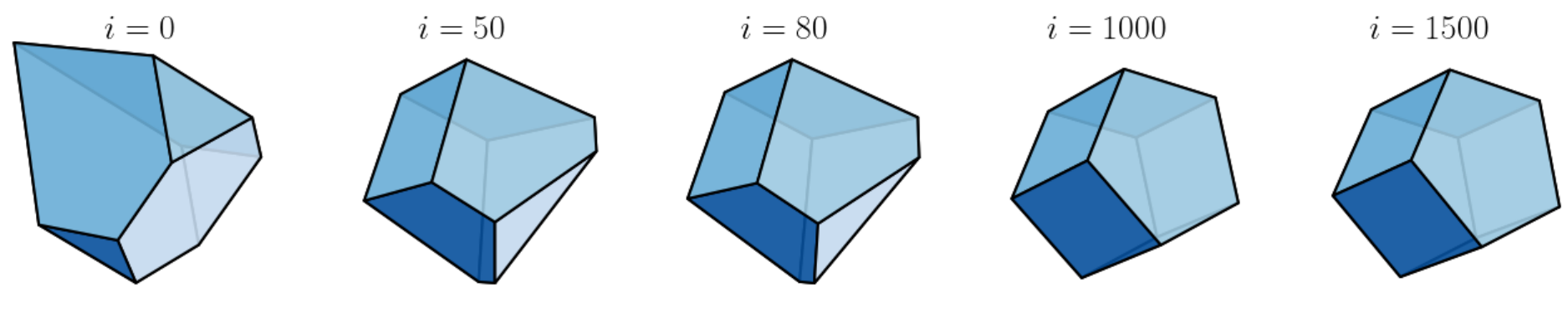}
	\vspace{0.1cm}
	\includegraphics[width=0.33\textwidth]{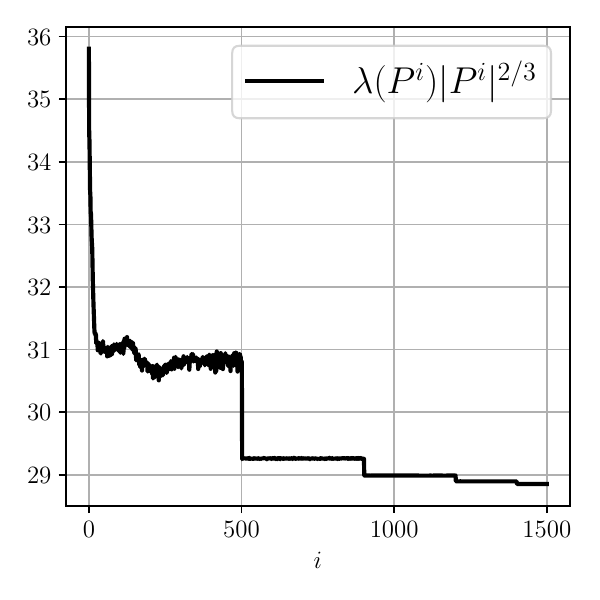}
	\hspace{0.2cm}
	\includegraphics[width=0.33\textwidth]{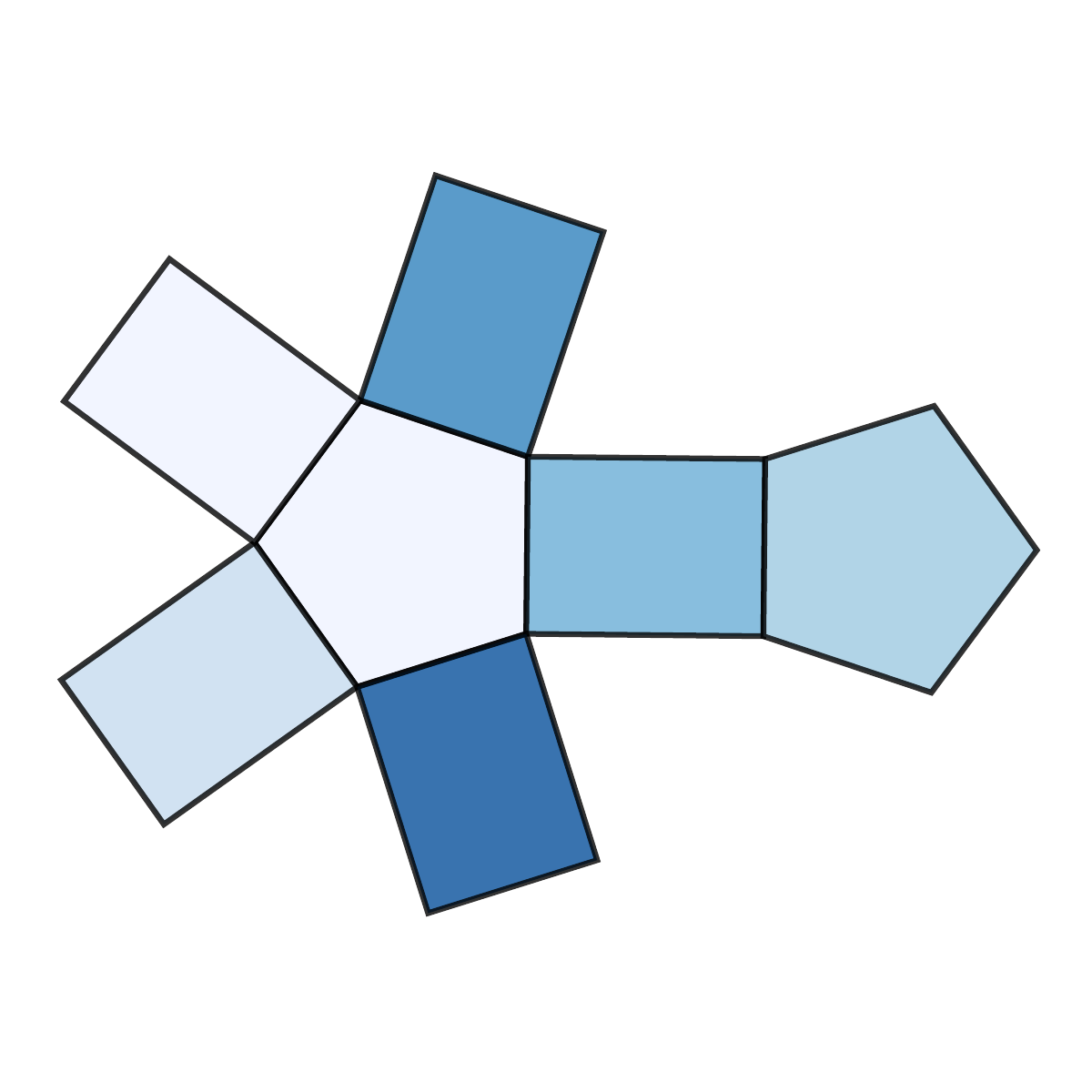}
	\caption{
		Case $\left|\mathcal{K}\right| = 7$.
		(Top)  Polyhedra corresponding to
		iterations $i=0,50,80,1000$, and $1500$.
		(Bottom) Values of $\lambda(P^i)\,|P^i|^{2/3}$ and
		the unfolding of the optimal polyhedron $P^{*}$.
		The initial polyhedron $P^{0}$ was
		a tetrahedron with three truncated vertices.
		The polyhedra $P^{i}$ for
		$i \geq 50$ were isomorphic to a pentagonal prism.
		The initial and final polyhedra are not isomorphic.
	}
	\label{fig:p07}
\end{figure}

\begin{figure}
	\centering
	\includegraphics[width=0.71\textwidth]{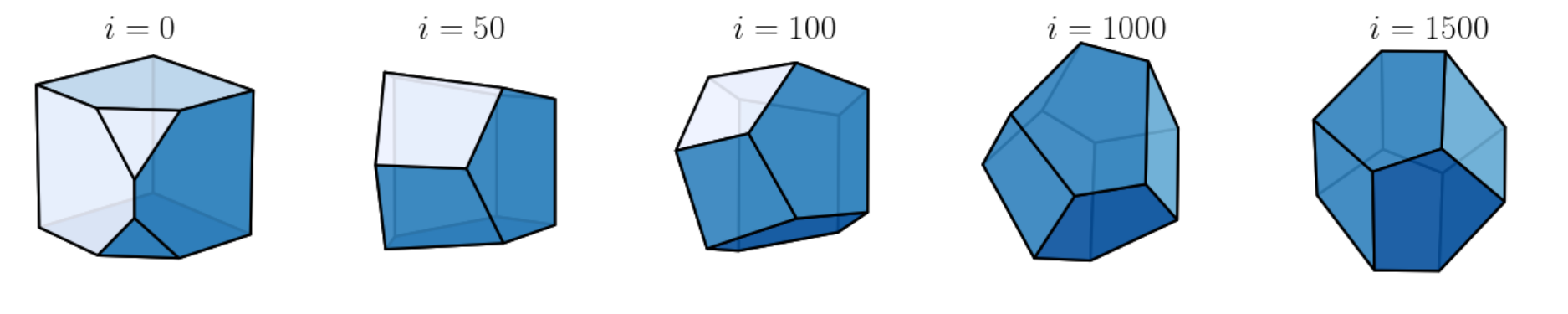}
	\vspace{0.1cm}
	\includegraphics[width=0.33\textwidth]{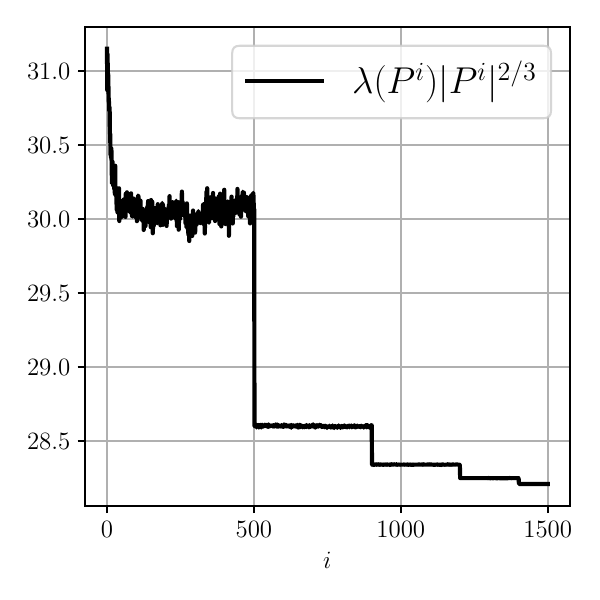}
	\hspace{0.2cm}
	\includegraphics[width=0.33\textwidth]{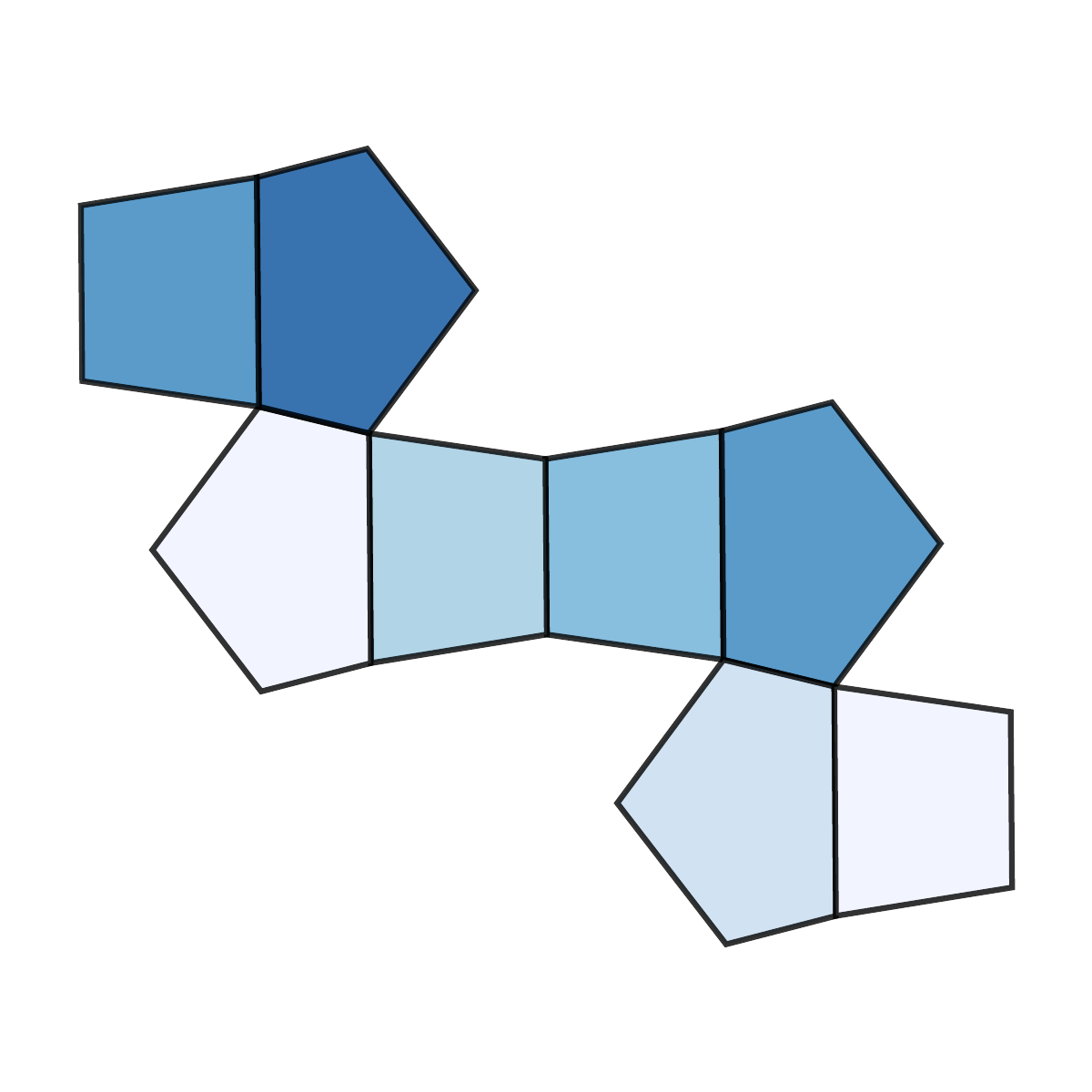}
	\caption{
		Case $\left|\mathcal{K}\right| = 8$.
		(Top) Polyhedra corresponding to
		iterations $i=0,50,100,1000$, and $1500$.
		(Bottom) Values of $\lambda(P^i)\,|P^i|^{2/3}$ and
		the unfolding of the optimal polyhedron $P^{*}$.
		The initial polyhedron $P^{0}$ was
		a cube with two truncated vertices.
		The polyhedra $P^{i}$ for
		$i \geq 50$ were isomorphic to a one another.
		The initial and final polyhedra are not isomorphic.
	}
	\label{fig:p08}
\end{figure}

\begin{figure}
	\centering
	\includegraphics[width=0.71\textwidth]{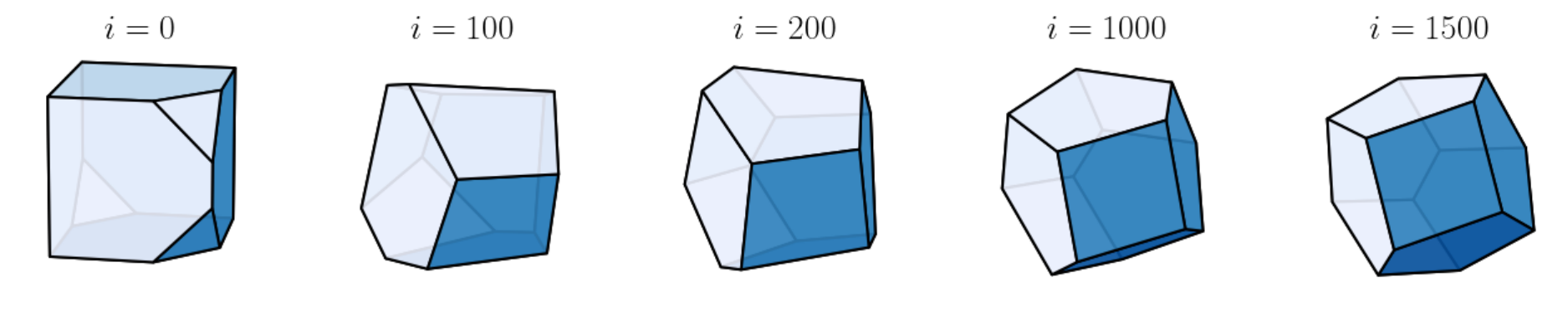}
	\vspace{0.1cm}
	\includegraphics[width=0.33\textwidth]{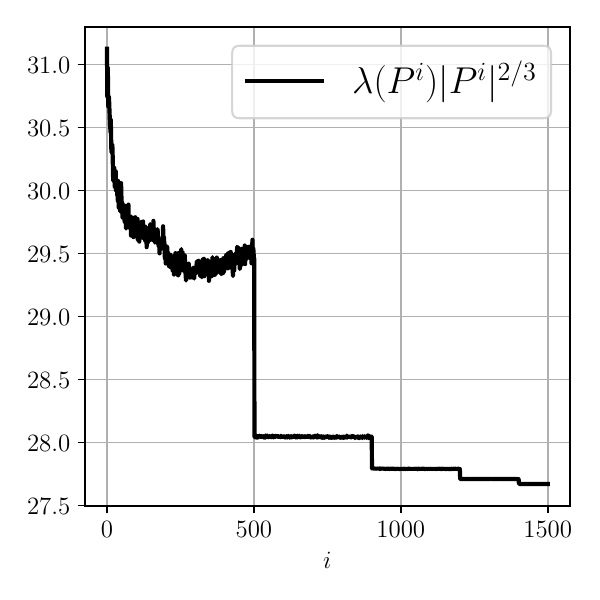}
	\hspace{0.2cm}
	\includegraphics[width=0.33\textwidth]{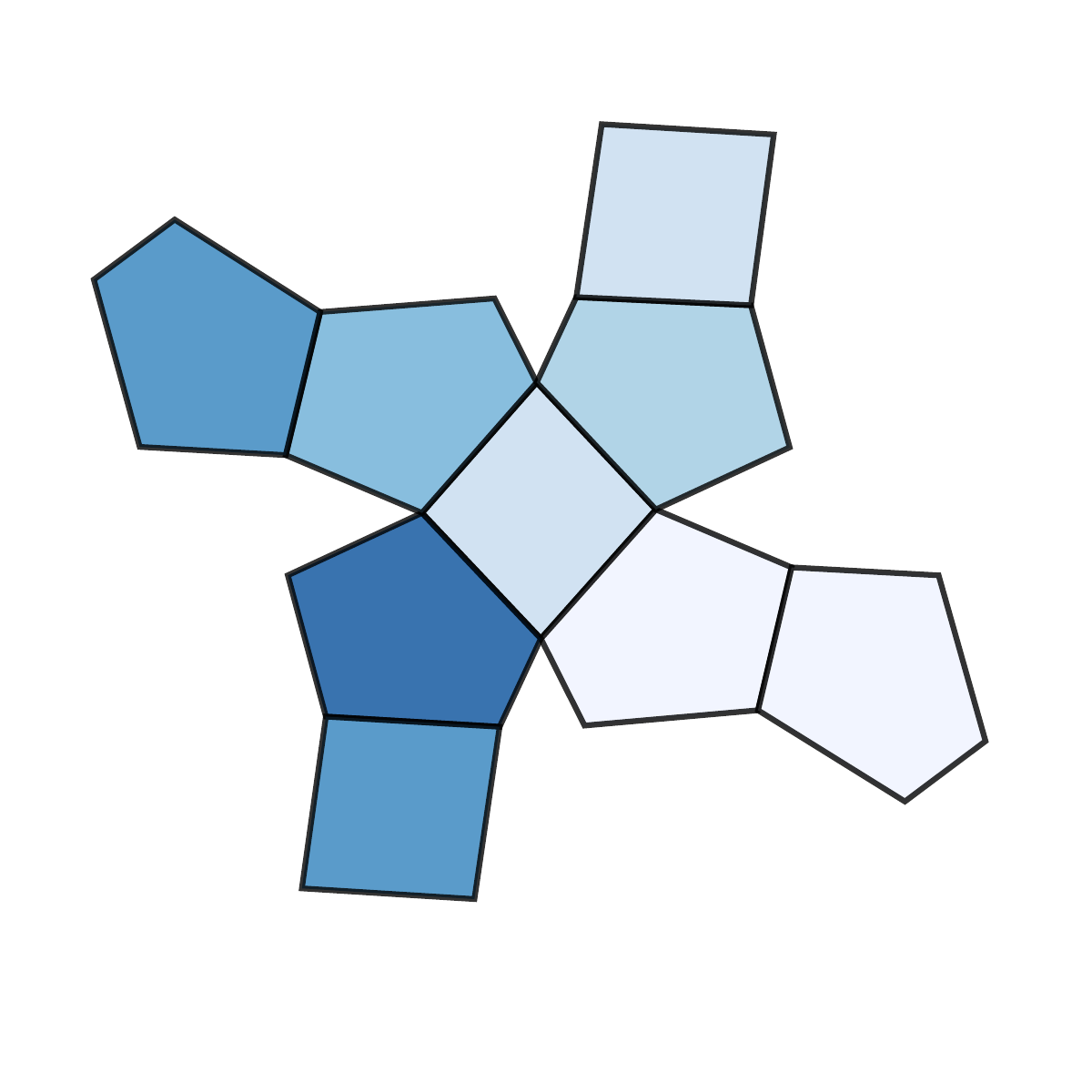}
	\caption{
		Case $\left|\mathcal{K}\right| = 9$.
		(Top) Polyhedra corresponding to
		iterations $i=0,100,200,1000$, and $1500$.
		(Bottom) Values of $\lambda(P^i)\,|P^i|^{2/3}$ and
		the unfolding of the optimal polyhedron $P^{*}$.
		The initial polyhedron $P^{0}$ was
		a cube with three truncated vertices.
		The polyhedra $P^{i}$ for
		$i \geq 200$ were isomorphic to a one another.
		The initial and final polyhedra are not isomorphic.
	}
	\label{fig:p09}
\end{figure}

\begin{figure}
	\centering
	\includegraphics[width=0.71\textwidth]{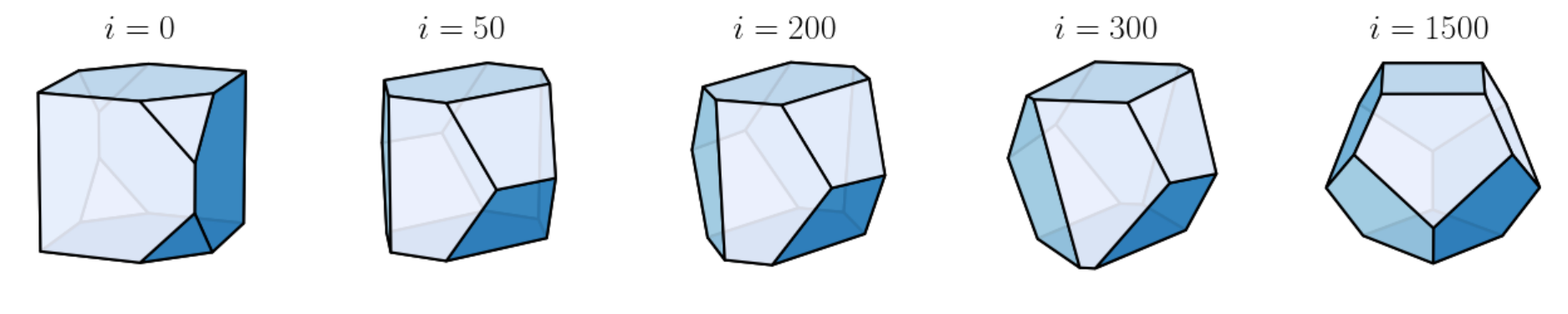}
	\vspace{0.1cm}
	\includegraphics[width=0.33\textwidth]{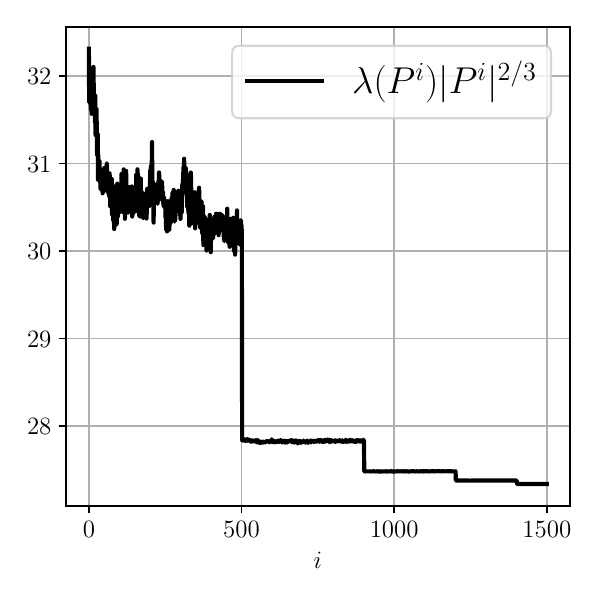}
	\hspace{0.2cm}
	\includegraphics[width=0.33\textwidth]{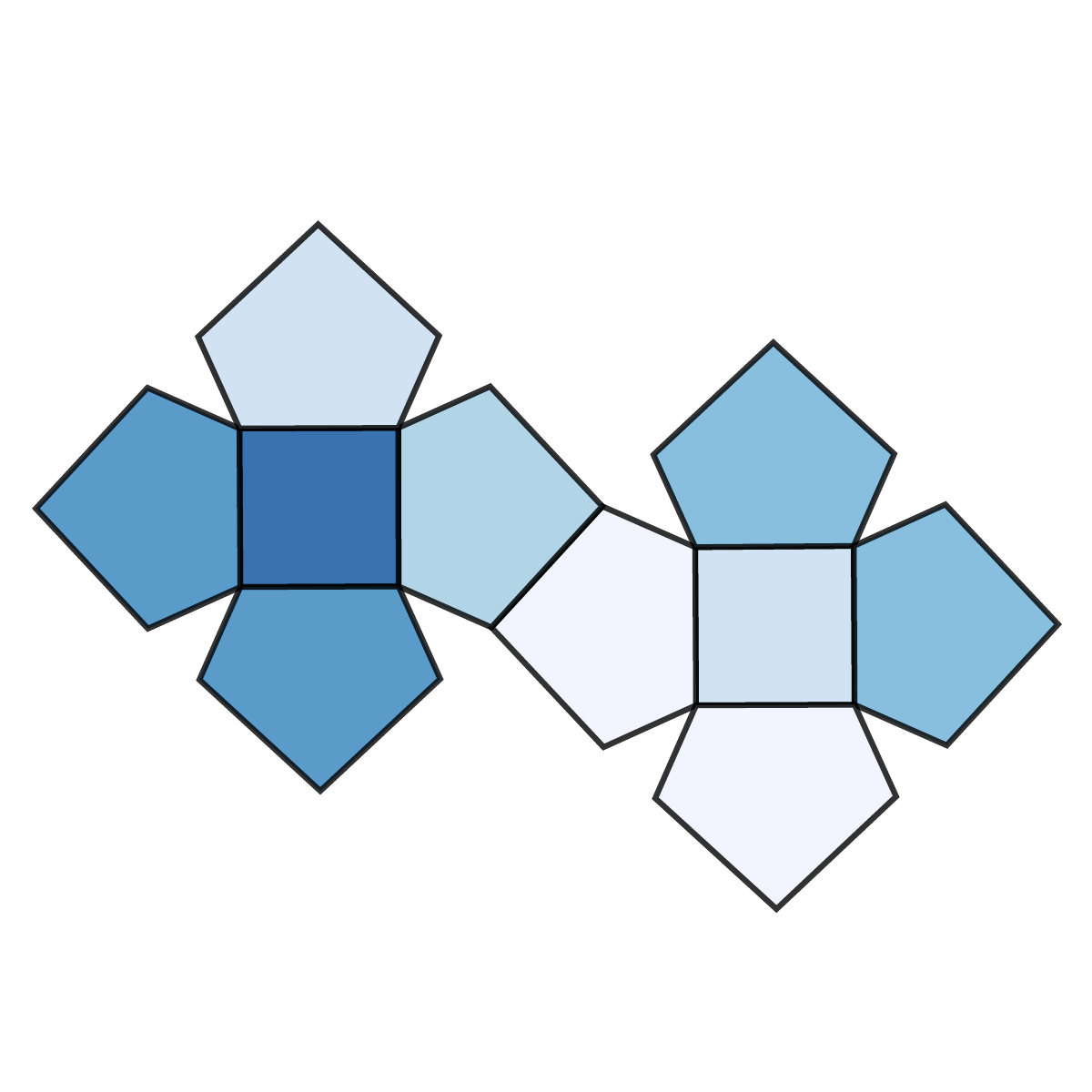}
	\caption{
		Case $\left|\mathcal{K}\right| = 10$.
		(Top) Polyhedra corresponding to
		iterations $i=0,50,200,300$, and $1500$.
		(Bottom) Values of $\lambda(P^i)\,|P^i|^{2/3}$ and
		the unfolding of the optimal polyhedron $P^{*}$.
		The initial polyhedron $P^{0}$ was
		a cube with four truncated vertices.
		The polyhedra $P^{i}$ for
		$i \geq 300$ were isomorphic to a one another.
		The initial and final domains are not isomorphic.
	}
	\label{fig:p10}
\end{figure}

\begin{figure}
	\centering
	\includegraphics[width=0.71\textwidth]{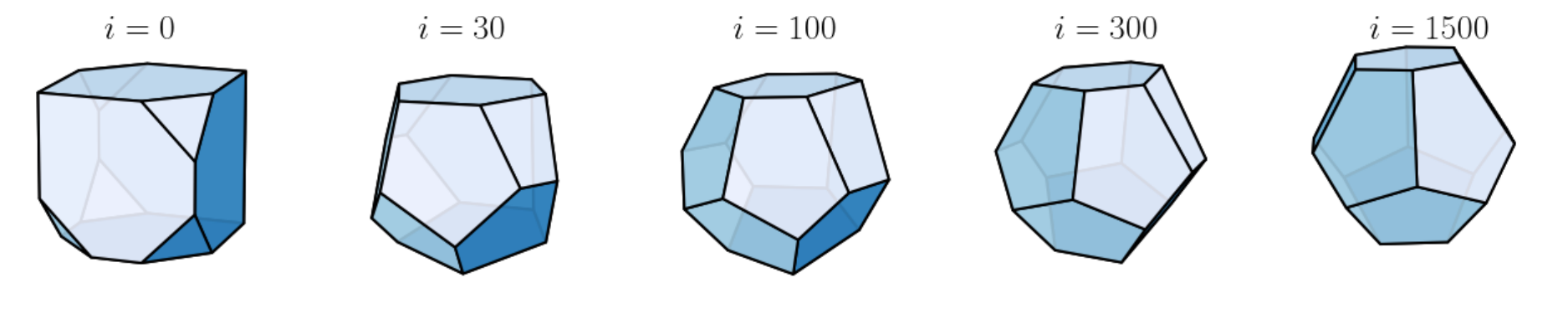}
	\vspace{0.1cm}
	\includegraphics[width=0.33\textwidth]{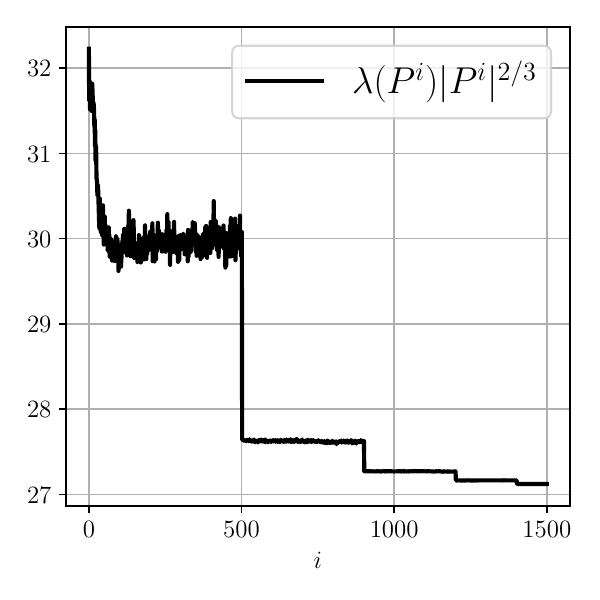}
	\hspace{0.2cm}
	\includegraphics[width=0.33\textwidth]{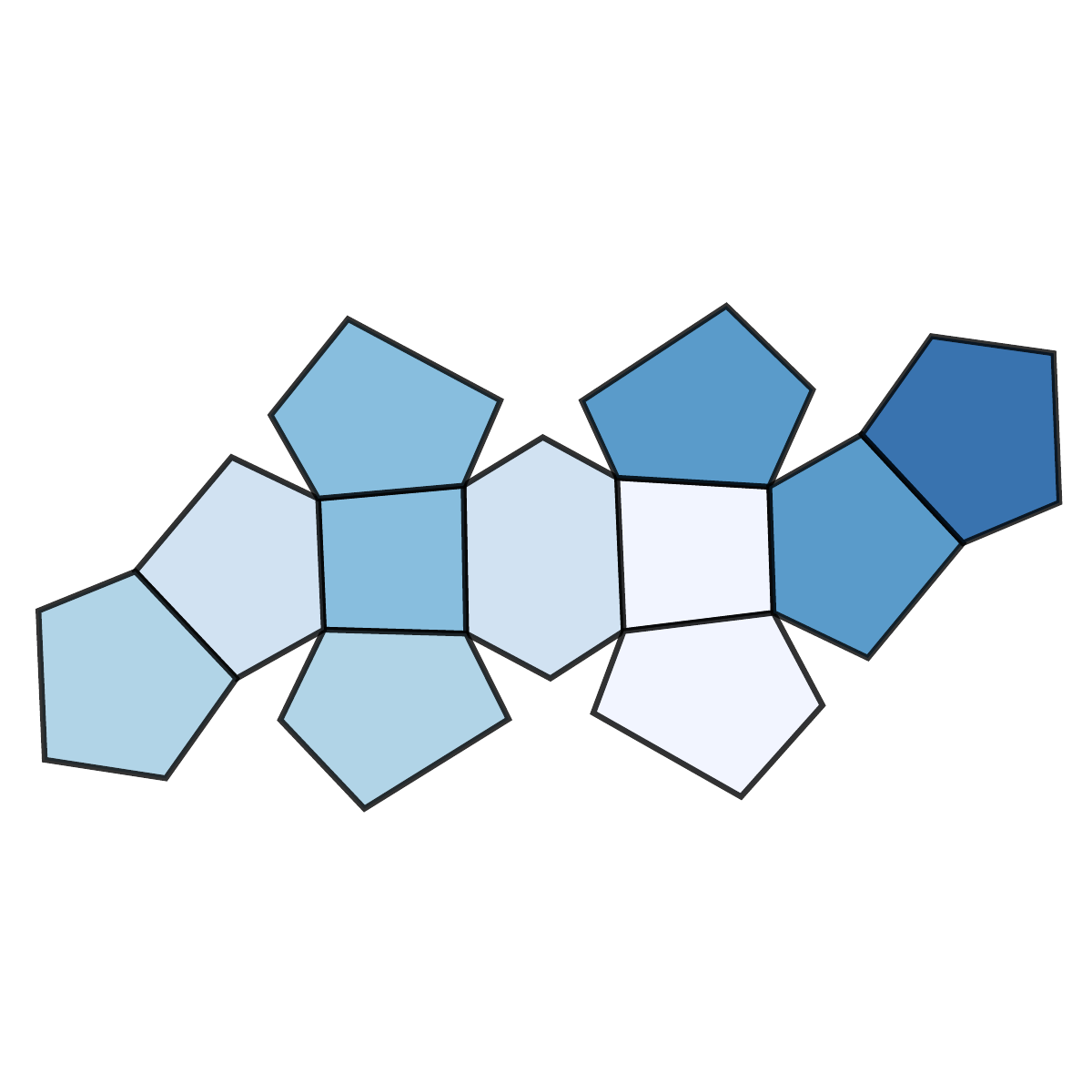}
	\caption{
		Case $\left|\mathcal{K}\right| = 11$.
		(Top) Polyhedra corresponding to
		iterations $i=0,30,100,300$, and $1500$.
		(Bottom) Values of $\lambda(P^i)\,|P^i|^{2/3}$ and
		the unfolding of the optimal polyhedron $P^{*}$.
		The initial polyhedron $P^{0}$ was
		a cube with five truncated vertices.
		The polyhedra $P^{i}$ for
		$i \geq 30$ were isomorphic to a one another.
		The initial and final domains are no isomorphic.
	}
	\label{fig:p11}
\end{figure}

\begin{figure}
	\centering
	\includegraphics[width=0.71\textwidth]{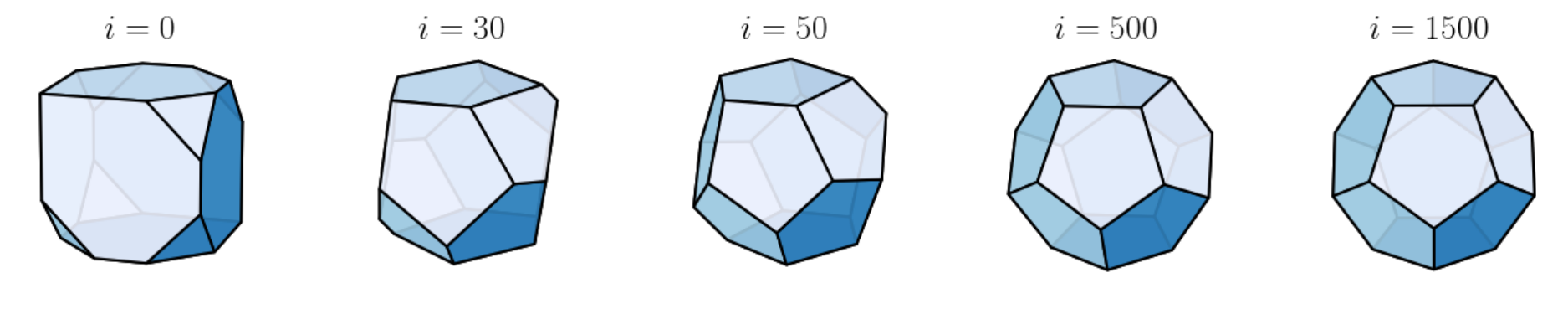}
	\vspace{0.1cm}
	\includegraphics[width=0.33\textwidth]{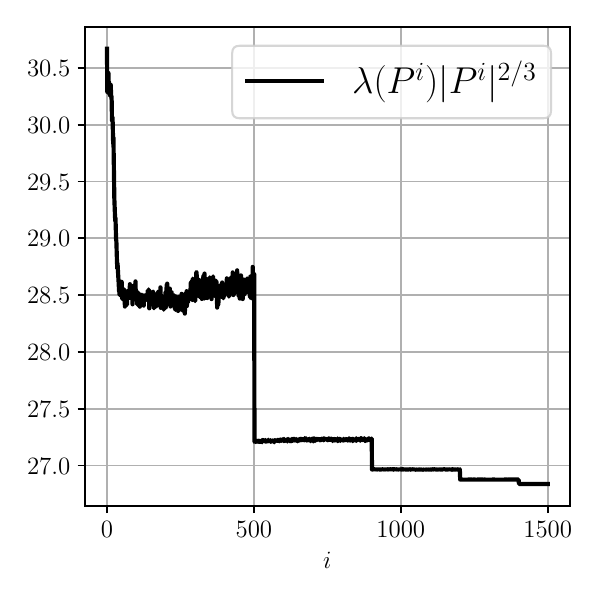}
	\hspace{0.2cm}
	\includegraphics[width=0.33\textwidth]{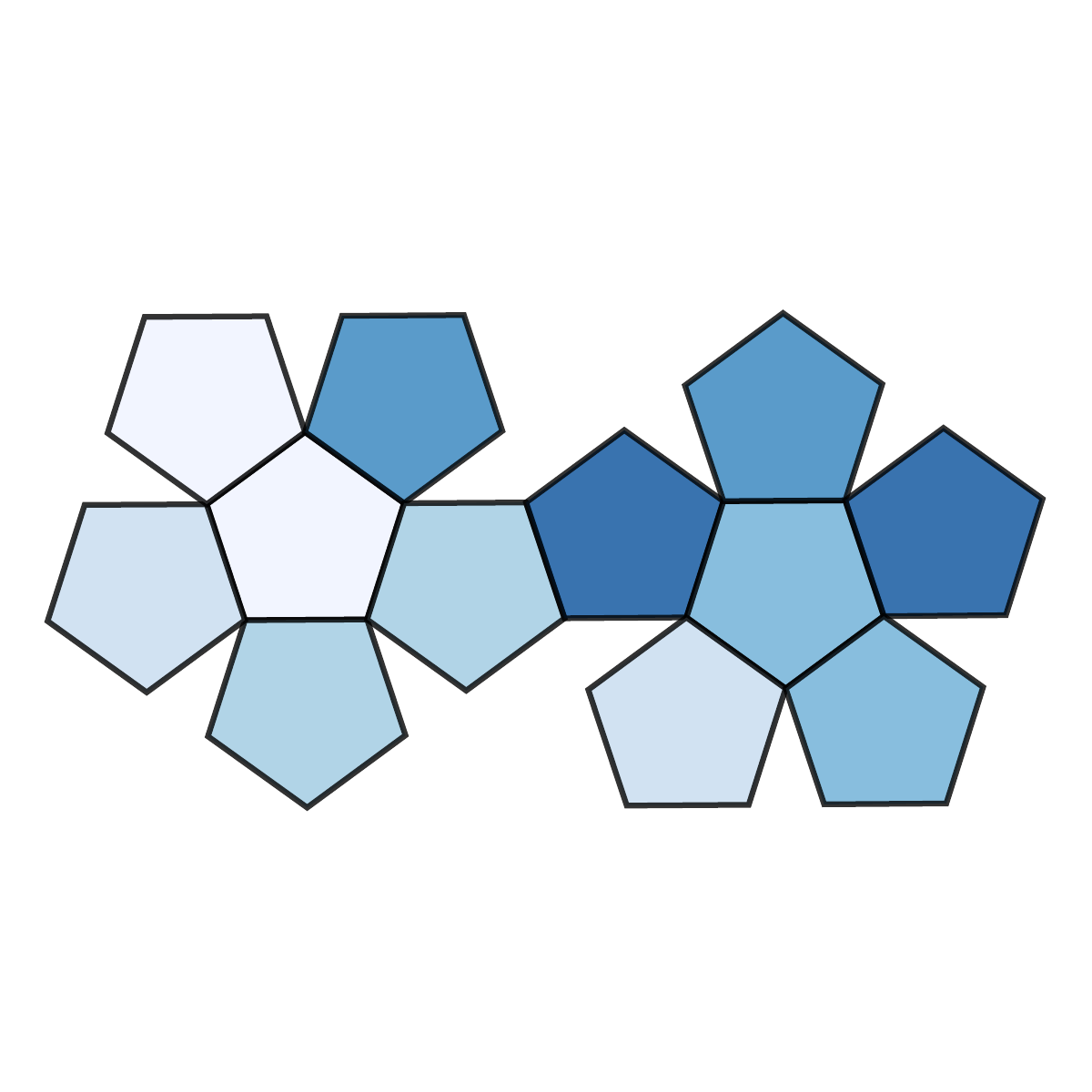}
	\caption{
		Case $\left|\mathcal{K}\right| = 12$.
		(Top) Polyhedra corresponding to
		iterations $i=0,30,50,500$, and $1500$.
		(Bottom) Values of $\lambda(P^i)\,|P^i|^{2/3}$ and
		the unfolding of the optimal polyhedron $P^{*}$.
		The initial polyhedron $P^{0}$ was
		a cube with six truncated vertices.
		The polyhedra $P^{i}$ for
		$i \geq 30$ were isomorphic to a dodecahedron.
		The initial and final domains are no isomorphic.
	}
	\label{fig:p12}
\end{figure}

\begin{figure}
	\centering
	\includegraphics[width=0.71\textwidth]{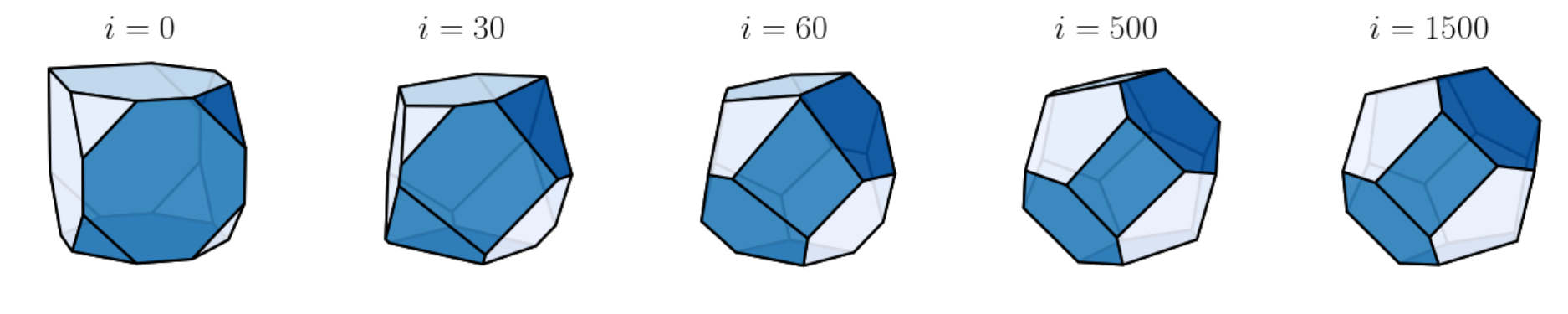}
	\vspace{0.1cm}
	\includegraphics[width=0.33\textwidth]{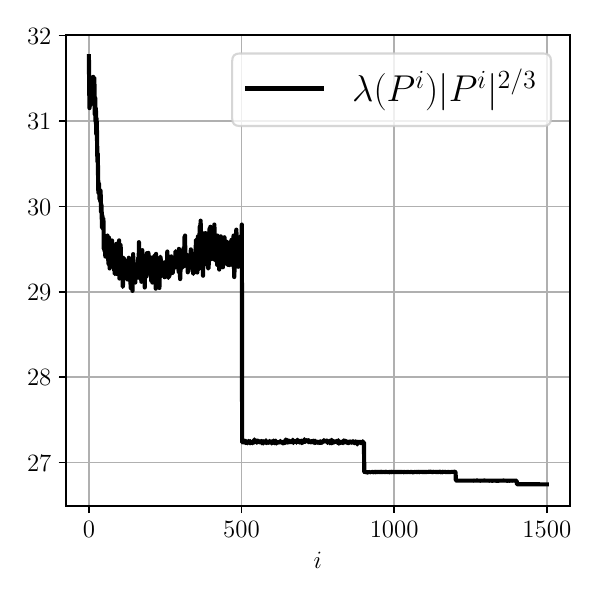}
	\hspace{0.2cm}
	\includegraphics[width=0.33\textwidth]{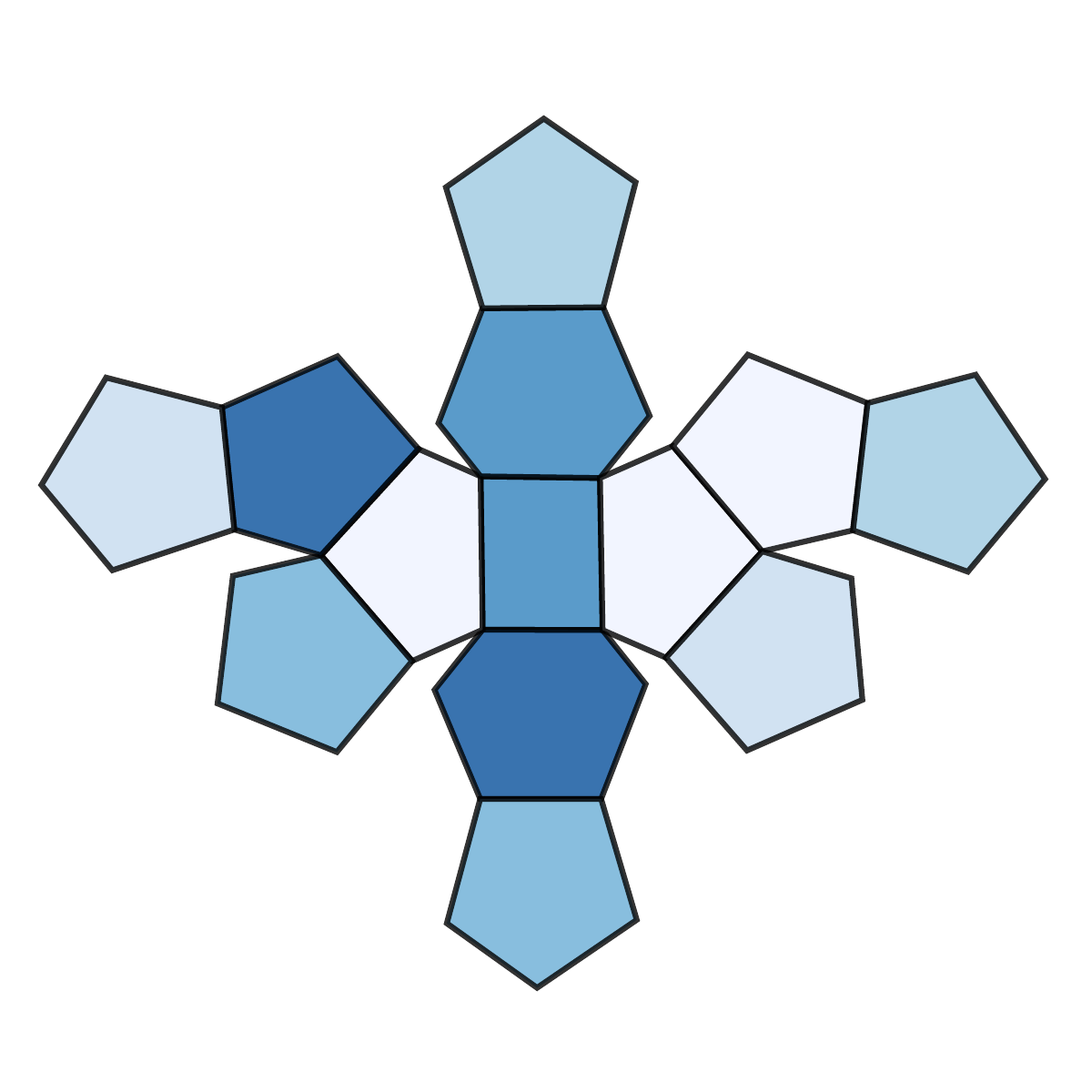}
	\caption{
		Case $\left|\mathcal{K}\right| = 13$.
		(Top) Polyhedra corresponding to
		iterations $i=0,30,60,500$ and $1500$.
		(Bottom) Values of $\lambda(P^i)\,|P^i|^{2/3}$ and
		the unfolding of the optimal polyhedron $P^{*}$.
		The initial polyhedron $P^{0}$ was
		a cube with seven truncated vertices.
		The polyhedra $P^{i}$ for
		$i \geq 100$ were isomorphic to a one another.
		The initial and final domains are no isomorphic.
	}
	\label{fig:p13}
\end{figure}

\begin{figure}
	\centering
	\includegraphics[width=0.71\textwidth]{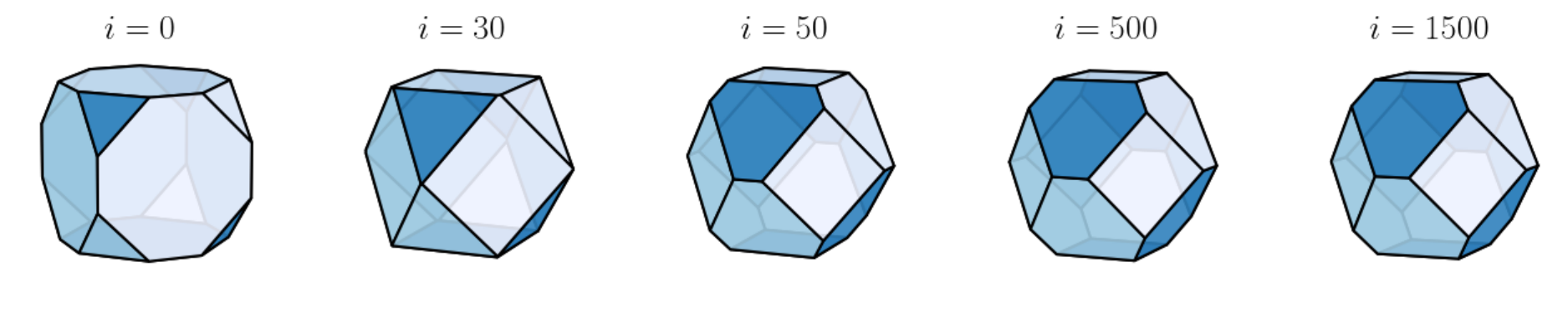}
	\vspace{0.1cm}
	\includegraphics[width=0.33\textwidth]{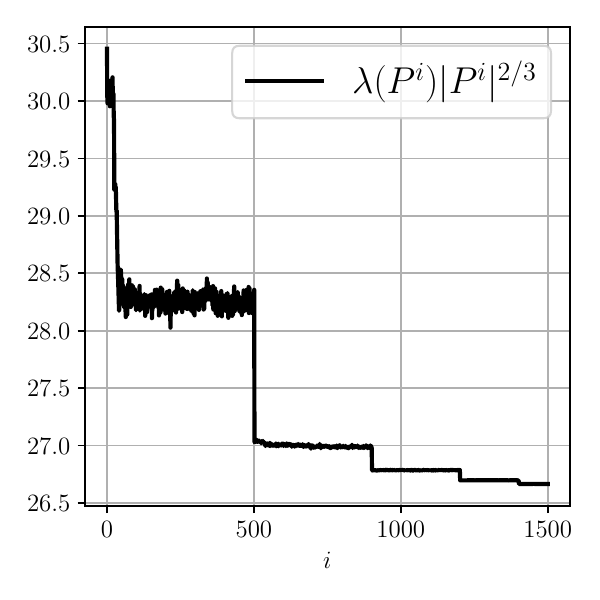}
	\hspace{0.2cm}
	\includegraphics[width=0.33\textwidth]{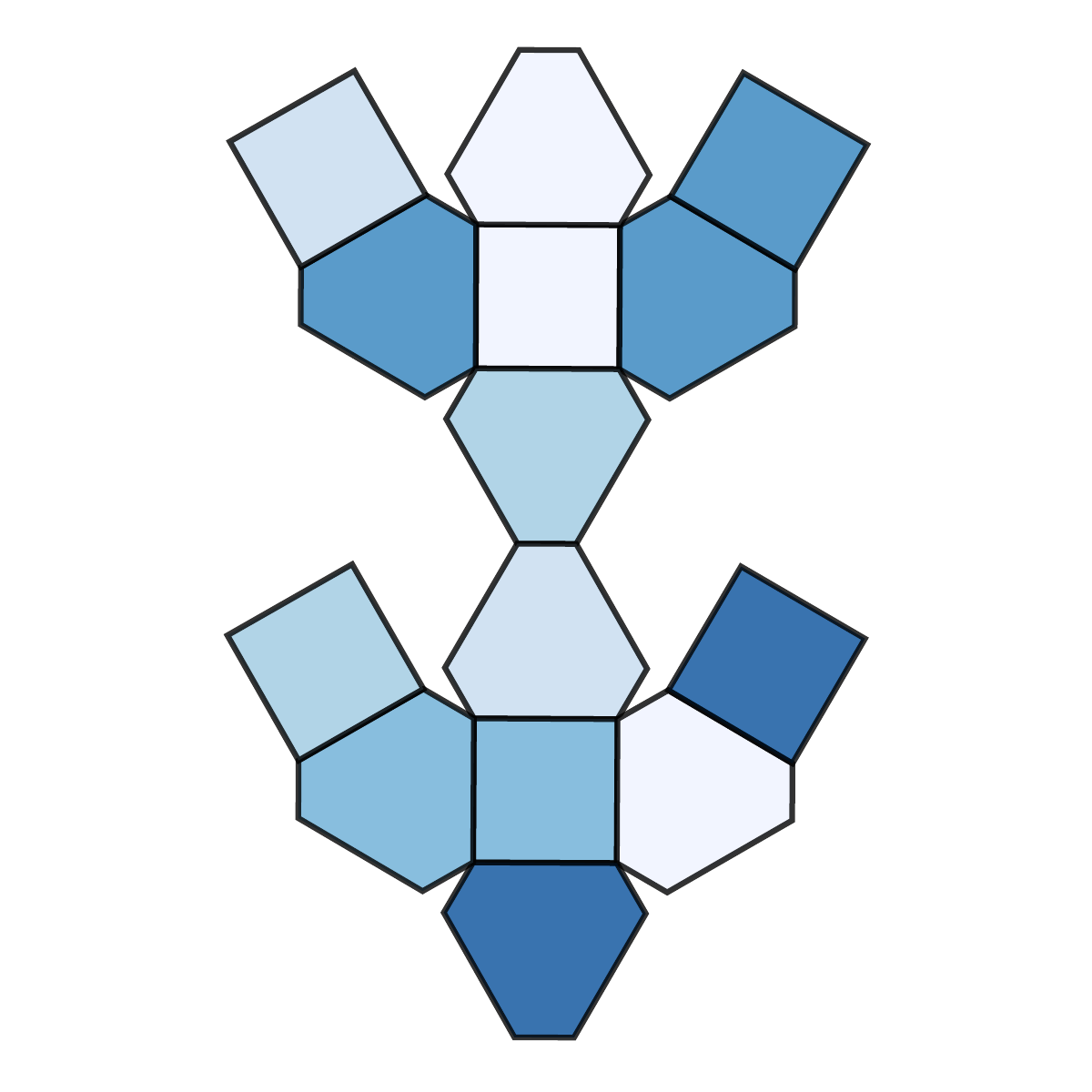}
	\caption{
		Case $\left|\mathcal{K}\right| = 14$.
		(Top) Polyhedra corresponding to
		iterations $i=0,30,50,500$, and $1500$.
		(Bottom) Values of $\lambda(P^i)\,|P^i|^{2/3}$ and
		the unfolding of the optimal polyhedron $P^{*}$.
		The initial polyhedron $P^{0}$ was
		a cube with all its truncated vertices.
		The polyhedra $P^{i}$ for
		$i \geq 40$ were isomorphic to a tetrakaidecahedron.
		The initial and final domains are no isomorphic.
	}
	\label{fig:p14}
\end{figure}

\begin{figure}
	\centering
	\includegraphics[width=0.9\textwidth]{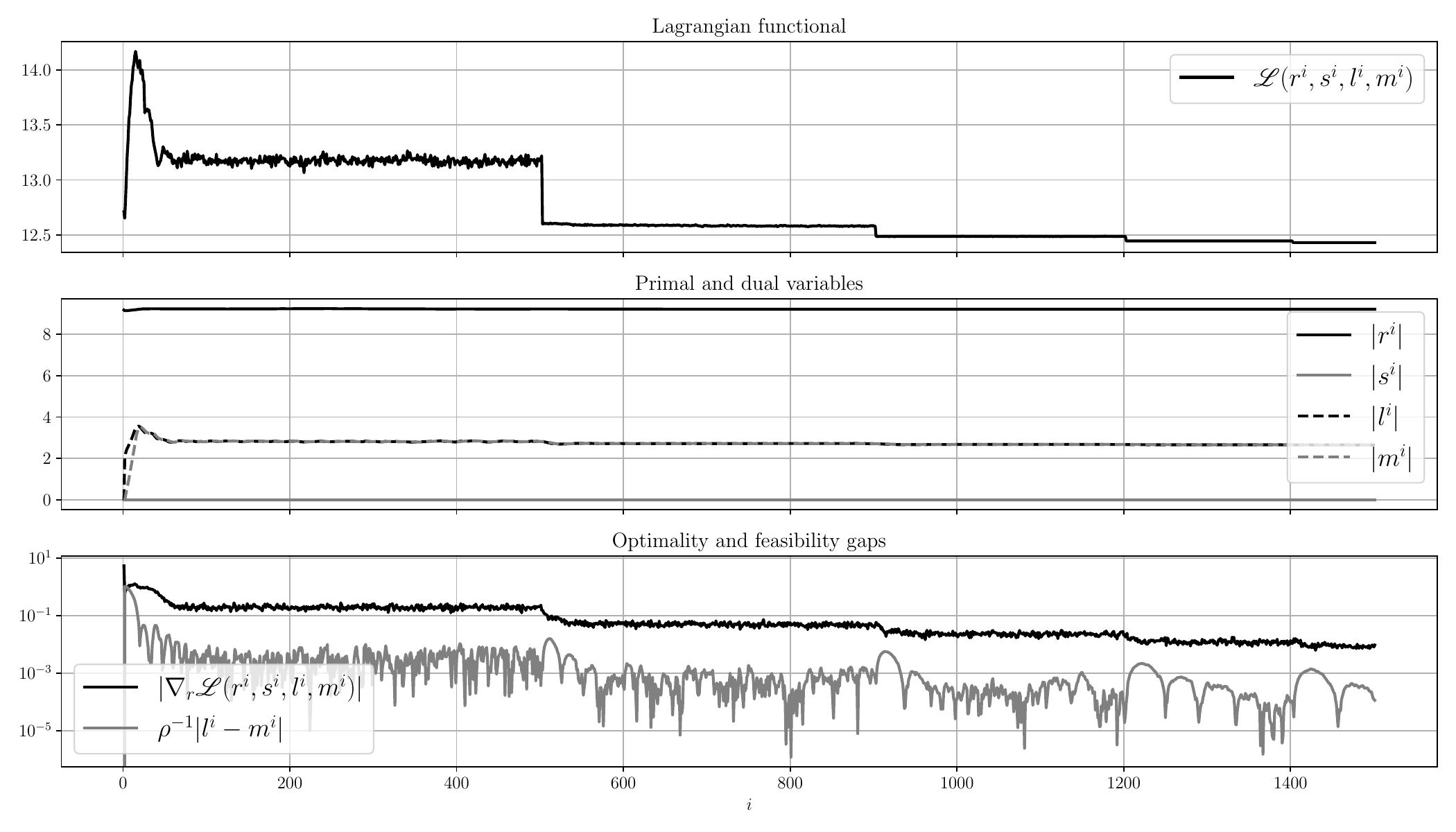}
	\caption{
		Values of $| \nabla_{r} \mathscr{L}(r^i,s^i,l^i,m^i) |$ (black) and
		$||P^i| - \pi|$ (gray) along the iterations for the test with $|\mathcal{K}|=14$ facets.
		These results show that algorithm~\eqref{eq:algo} 
		produces a decrease in the gradient norm and 
		enforces the volume constraint.
	}\label{fig:gaps}
\end{figure}

\clearpage

\appendix

\section{GeoGebra commands for reproducing the optimal polyhedra}

The following links provide \texttt{GeoGebra} commands for reproducing
and visualizing the optimal polyhedra.
For instance, in the case with $\left|\mathcal{K}\right| = 5$ facets,
see \url{https://www.geogebra.org/3d/xet5kbgg} for the optimal polyhedron,
together with its vertices and facets:
\begin{figure}[h]
	\centering
	\includegraphics[width=0.35\textwidth]{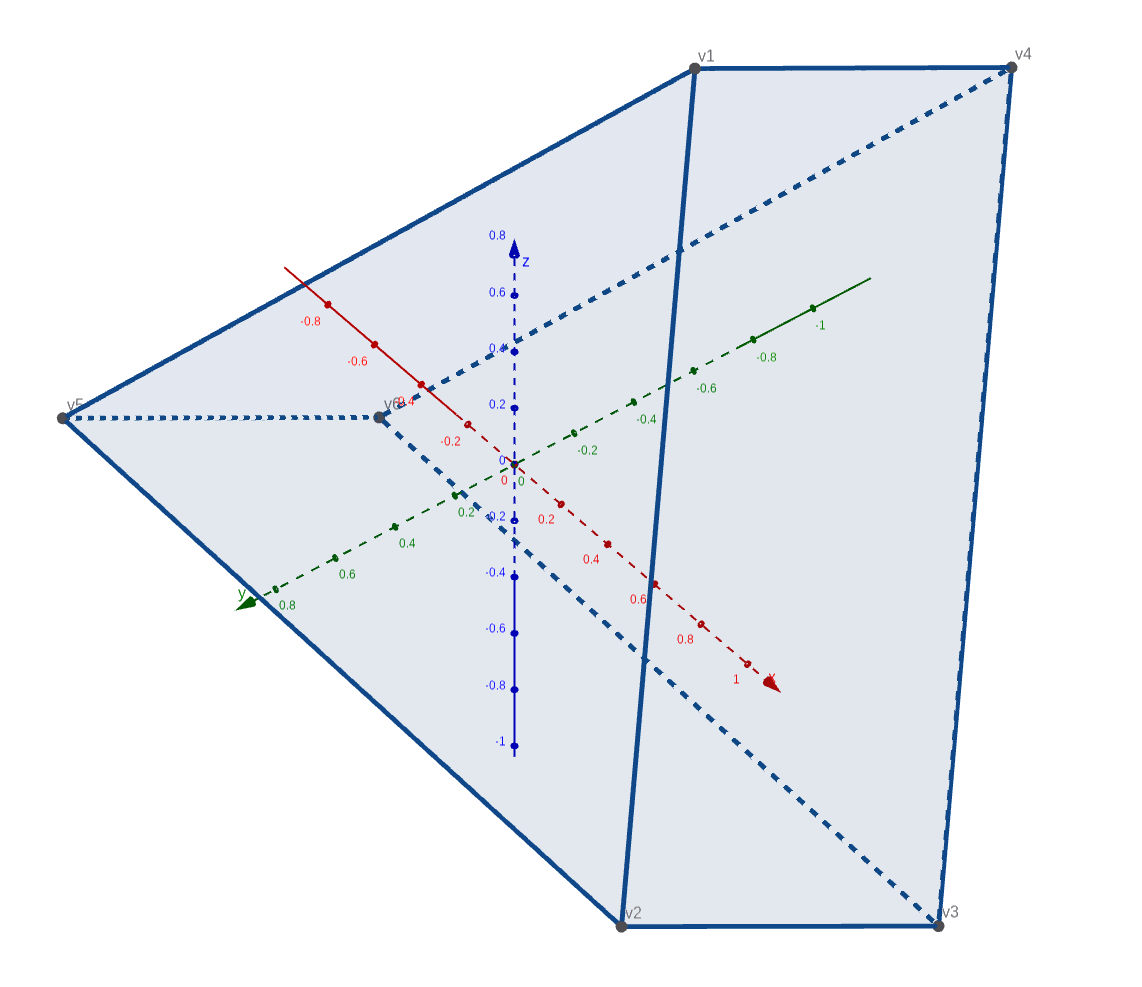}
\end{figure}
\begin{lstlisting}
v1=Point({-1.11202959978901,0.702097019822660,-0.171572502895451})
v2=Point({-1.12456315042285,-0.762271965502623,-0.197040948792147})
v3=Point({-0.066183819257264,0.711108449634161,-1.196613428820373})
v4=Point({-0.078544744866793,-0.753254337976977,-1.222250944292790})
v5=Point({-0.099202288482267,-0.789603080025962,0.848108864198090})
v6=Point({-0.086338699585014,0.674747670706212,0.873913636147723})
v7=Point({0.959366251944387,0.683760623392674,-0.151144332933853})
v8=Point({0.946675295578225,-0.780583930144844,-0.177118205313624})
f1 = Polygon({v5, v6, v7, v8})
f2 = Polygon({v1, v2, v5, v6})
f3 = Polygon({v2, v4, v8, v5})
f4 = Polygon({v1, v3, v7, v6})
f5 = Polygon({v1, v2, v4, v3})
f6 = Polygon({v3, v4, v8, v7})
\end{lstlisting}

The remaining cases are listed below:
\begin{itemize}
	\item $\left|\mathcal{K}\right| = 4$ facets, see \url{https://www.geogebra.org/3d/acadmuzp};
	\item $\left|\mathcal{K}\right| = 6$ facets, see \url{https://www.geogebra.org/3d/mrwp6nqb};
	\item $\left|\mathcal{K}\right| = 7$ facets, see \url{https://www.geogebra.org/3d/nrnyvwbj};
	\item $\left|\mathcal{K}\right| = 8$ facets, see \url{https://www.geogebra.org/3d/xj2vem62};
	\item $\left|\mathcal{K}\right| = 9$ facets, see \url{https://www.geogebra.org/3d/te7kjbbq};
	\item $\left|\mathcal{K}\right| = 10$ facets, see \url{https://www.geogebra.org/3d/q8skk9cn};
	\item $\left|\mathcal{K}\right| = 11$ facets, see \url{https://www.geogebra.org/3d/xayqz6te};
	\item $\left|\mathcal{K}\right| = 12$ facets, see \url{https://www.geogebra.org/3d/cwnvvrwa};
	\item $\left|\mathcal{K}\right| = 13$ facets, see \url{https://www.geogebra.org/3d/p36zg6fn};
	\item $\left|\mathcal{K}\right| = 14$ facets, see \url{https://www.geogebra.org/3d/haf2yauv}.
\end{itemize}

\clearpage

\bibliographystyle{plain}
\bibliography{eigen.bib}

\end{document}